\documentclass[11pt]{amsart}
\usepackage{times}
\usepackage[T1]{fontenc}
\usepackage{amssymb, amsthm, amsmath}
\usepackage[active]{srcltx}
\usepackage{hyperref}
\usepackage{setspace}
\usepackage{geometry}
\DeclareMathOperator{\acl}{acl}
\DeclareMathOperator{\dcl}{dcl}

 \DeclareMathOperator{\theo}{Th}

\DeclareMathOperator{\tp}{tp}\DeclareMathOperator{\stp}{stp}

\DeclareMathOperator{\mr}{RM}

\DeclareMathOperator{\rk}{rk}

\newtheorem{introtheorem}{Theorem}

\newtheorem{theorem}{Theorem}[section]

\newtheorem{claim}[theorem]{Claim}

\newtheorem{conjecture}[theorem]{Conjecture}
\newtheorem{corollary}[theorem]{Corollary}

\newtheorem{fact}[theorem]{Fact}
\newtheorem{lemma}[theorem]{Lemma}

\newtheorem{proposition}[theorem]{Proposition}

\theoremstyle{definition}
\newtheorem{definition}[theorem]{Definition}

\newtheorem{remark}[theorem]{Remark}

\newtheorem{notation}[theorem]{Notation}
\newtheorem{convention}[theorem]{Convention}

\newcommand{\Rr}{{\mathbb{R}}}
\newcommand{\Nn}{{\mathbb{N}}}

\newcommand{\CL}{{\mathcal L}}
\newcommand{\CK}{{\mathcal K}}
\newcommand{\CN}{{\mathcal N}}

\newcommand{\CR}{{\mathcal R}}
\newcommand{\CM}{{\mathcal M}}

\newcommand{\CC}{{\mathcal C}}

\newcommand{\CG}{{\mathcal G}}
\newcommand{\CS}{\mathcal S}

\newcommand{\0}{\emptyset}

\renewcommand{\phi}{\varphi}

\long\def\symbolfootnote[#1]#2{\begingroup%
\def\thefootnote{\fnsymbol{footnote}}\footnote[#1]{#2}\endgroup}

\makeatletter

\def\Ind#1#2{#1\setbox0=\hbox{$#1x$}\kern\wd0\hbox to 0pt{\hss$#1\mid$\hss}
\lower.9\ht0\hbox to 0pt{\hss$#1\smile$\hss}\kern\wd0}

\def\Notind#1#2{#1\setbox0=\hbox{$#1x$}\kern\wd0\hbox to 0pt{\mathchardef
\nn=12854\hss$#1\nn$\kern1.4\wd0\hss}\hbox to
0pt{\hss$#1\mid$\hss}\lower.9\ht0 \hbox to
0pt{\hss$#1\smile$\hss}\kern\wd0}

\def\sub{\subseteq}

\makeatother

\title{}

\date{\today}

\begin{document}


 \oddsidemargin .8cm
 \evensidemargin .8cm

\title{Strongly Minimal Relics of T-convex Fields}

\date{\today}
\author{Benjamin Castle}
\address{Department of Mathematics, University of Illinois Urbana-Champaign, Urbana, IL 61820}
\email{btcastl2@illinois.edu }

\author{Assaf Hasson}
\address{Department of Mathematics, Ben Gurion University of the Negev, Be'er-Sheva 84105, Israel}
\email{hassonas@math.bgu.ac.il}

\setcounter{tocdepth}{1}

 \thanks{The first author was partially supported by a Brin postdoctoral fellowship at the University of Maryland. The second author was supported by ISF grant No. 555/21.}

 \begin{abstract} Generalizing previous work on algebraically closed valued fields (ACVF) and o-minimal fields, we study strongly minimal relics of real closed valued fields (RCVF), and more generally T-convex expansions of o-minimal fields. Our main result (replicating the o-minimal setting) is that non-locally modular strongly minimal definable relics of T-convex fields must be two-dimensional. We also continue our work on reducing the trichotomy for general relics of a structure to just the relics of certain distinguished sorts. To this end, we prove that the trichotomy for definable RCVF-relics implies the trichotomy for interpretable RCVF-relics, and also that the trichotomy for relics of o-minimal fields implies the trichotomy for relics of any dense o-minimal structure. Finally, we introduce the class of \textit{differentiable Hausdorff geometric fields} (containing o-minimal fields and various valued fields), and give a general treatment of the trichotomy for one-dimensional relics of such fields (namely, reducing the trichotomy for one-dimensional relics to an axiomatic condition on the field itself).
 \end{abstract}

 \maketitle
 \tableofcontents{\setcounter{tocdepth}{0}}

\section{Introduction}
We continue our work on Zilber's trichotomy in Hausdorff geometric structures. These structures were introduced in \cite{CaHaYe} in order to complete the proof of the trichotomy for relics of algebraically closed fields (recall that a \textit{relic} of a structure $\CK$ is another structure whose definable sets are all interpretable sets in $\CK$). The idea was that in positive characteristic, it was easier to work with relics of algebraically closed \textit{valued fields} (ACVF) -- and Hausdorff geometric structures capture the main topological properties one needs in this setting. The result was a proof of the trichotomy for strongly minimal \textit{definable} ACVF-relics (see Convention \ref{Con: def vs int}), solving a conjecture of Kowalski and Randriambololona \cite{KowRand}.

In \cite{CasOmin}, the first author then applied the machinery of Hausdorff geometric structures in the o-minimal context, proving that if $\mathcal R$ is an o-minimal expansion of a field, and $\CM$ is a strongly minimal $\mathcal R$-relic on a universe of dimension greater than 2, then $\CM$ is locally modular. This is a special case of an older conjecture of Peterzil: 
\begin{conjecture}[Peterzil]
    Let $\CR$ be an o-minimal structure. A strongly minimal $\CR$-relic is either locally modular or interprets an algebraically closed field (that is, the Zilber trichotomy holds for strongly minimal relics of o-minimal strucures). 
\end{conjecture}
    
Given the above successes in the contexts of valued fields and o-minimal fields, it seems natural to try to merge them. Thus, in this paper, we study strongly minimal relics of $T$-\textit{convex fields} (i.e. o-minimal fields with a certain natural valuation). These structures were introduced in \cite{vdDriesLewen}. To construct them, one starts with an o-minimal theory $T$ expanding the theory of real closed fields, and then forms a new theory $T_{conv}$ by adding a predicate for a valuation ring, which is required to be convex (in the order) and closed under all global continuous $\emptyset$-definable functions. The simplest example is RCVF, the theory of real closed valued fields (which is precisely the theory $T_{conv}$ for $T=RCF$). Note that given an o-minimal theory $T$ of fields, we will use notation such as $\CR_V$ for models of $T_{conv}$, reserving $\CR$ for the reduct of $\CR_V$ to $T$.

We now move toward stating our main result (Theorem \ref{T: at most 2} and Theorem \ref{T: at least 2}). As in the case of ACVF, the main theorem will cover \textit{definable} relics: 

\begin{convention}\label{Con: def vs int} Throughout this paper, given a structure $\mathcal N$, we distinguish between \textit{definable} sets in $\CN$ (which are subsets of powers $N^k$) and \textit{interpretable} sets in $\CN$ (which allow for quotients by definable equivalence relations). We use the term \textit{$\CN$-relic} for \textit{interpretable relics} (i.e. with universe an $\CN$-inerpretable set), and we use \textit{definable $\CN$-relic} for relics with universe an $\CN$-definable set.
\end{convention}

In this language, our main result is:

\begin{introtheorem}\label{T: main}
    Let $T$ be an o-minimal theory of fields, and $T_{conv}$ its $T$-convex expansion. Let $\CR_V\models T_{conv}$, and let $\CM=(M,...)$ be a strongly minimal non-locally modular definable $\CR_V$-relic. Then $\dim(M)=2$.  
\end{introtheorem}

In particular, this theorem extends the main results of \cite{CasOmin} (eliminating the case $\dim(M)>2$) and \cite{HaOnPe} (eliminating the case $\dim(M)=1$) from the o-minimal setting to the $T$-convex setting. In the former case (eliminating $\dim(M)>2$), the proof is done by a reduction to the o-minimal setting. In the latter case (eliminating $\dim(M)=1$), we give a new proof, thus reproducing the main results of \cite{HaOnPe}. 


In the special case that $T_{conv}=RCVF$, our theorem also covers \textit{general} (i.e. interpretable) relics. The result (Theorem \ref{T: RCVF}) is that a non-locally modular strongly minimal $\CR_V$-relic is elementarily equivalent to a definable relic (precisely, the original relic is either itself definable, or it embeds into a power of the residue field, which can be expanded to a model of $T_{conv}$ in its own right). The proof of this result uses techniques from \cite[\S 12]{CaHaYe}, based on ideas from \cite{HaHaPe} (see Theorem \ref{T: intro weak rank} below). 

Along the way, we prove a few other results that may be of interest. First, in section 2 we prove a general result on reducing the trichotomy for relics of a given structure to relics of specific sorts:

\begin{introtheorem}\label{T: intro weak rank}
    Let $\CK$ be a weakly ranked structure, and $\CM=(M,...)$ a non-locally modular strongly minimal $\CK$-relic. If $M$ is locally interalgebraic with a $\CK$-definable set $D$, then $M$ is almost internal to $D$. 
\end{introtheorem}

Here \emph{weakly ranked} (Definition \ref{D: weakly ranked}) is an \emph{ad hoc} terminology aimed to capture various finite rank notions (e.g. $\acl$-rank in geometric structures, finite burden, finite SU-rank, etc.); while \textit{local interalgebraicity} of $M$ and $D$ just means the existence of an interpretable finite correspondence between infinite subsets of $M$ and $D$.

Theorem \ref{T: intro weak rank} generalizes our work from \cite[\S 12]{CaHaYe} in ACVF (and the proof works exactly as in that paper). Theorem \ref{T: intro weak rank} is then our main tool in treating interpretable sets in RCVF as discussed above. The idea is that by results of \cite{HaHaPe}, Theorem \ref{T: intro weak rank} reduces the trichotomy for \textit{all} RCVF-relics to just the relics of four distinguished sorts -- and two of these can be ruled out by virtue of being `1-based' in some sense. Thus one only has to consider the remaining two sorts -- and these are just the valued field itself and the residue field (i.e. exactly what we wish to have in the end).

As an additional application of Theorem \ref{T: intro weak rank}, in Section 3 we show that Peterzil's conjecture reduces to o-minimal expansions of fields (Theorem \ref{T: reduction2fields}). This is a useful fact to verify, as much of the progress on Peterzil's conjecture has been restricted to expansions of fields: 

\begin{introtheorem}\label{T: intro omin}
    Let $\CR$ be a dense o-minimal structure. Let $\CM$ be a non-locally modular strongly minimal $\CR$-relic. Then $M$ is internal to an interval in $R$ admitting an $\CR$-definable real closed field structure. 
\end{introtheorem}

Theorem \ref{T: intro omin} is proved by combining Theorem \ref{T: intro weak rank} with the Peterzil-Starchenko o-minimal trichotomy \cite{PeStTricho}. So here the `four distinguished sorts' appearing above in RCVF are replaced by the three basic types of intervals appearing in o-minimal structures. 

Sections 4-11 constitute the proof of Theorem \ref{T: main}, with sections 4-9 treating the case $\dim(M)>2$ and sections 10-11 treating the case $\dim(M)=1$. Before moving on, then, let us point out that sections 10-11 are done in a more general context. In the end, we work with the notion of a \textit{differentiable Hausdorff geometric field} (Definition \ref{D: differentiable HGF}). This is a geometric expansion of a field (typically of characteristic zero) equipped with a field topology and suitable notions of smoothness and tangent spaces. Examples include Henselian valued fields of characteristic zero, as well as o-minimal fields and their T-convex expansions.

In this language, we isolate a general condition called \textit{TIMI} (Definition \ref{D: TIMI}, adapted slightly from the same notion in \cite{CaHaYe}), roughly stating that `tangent intersections' of curves should also be `multiple intersections' in a topological sense. In Theorem \ref{T: t-convex timi}, we prove that T-convex fields satisfy TIMI. We then prove (Theorem \ref{T: field interpretation}) that if $(\CK,\tau)$ is a differentiable Hausdorff geometric field satisfying TIMI, then the trichotomy holds for all strongly minimal definable $\CK$-relics on universes of dimension 1:

\begin{introtheorem}\label{T: intro field interpretation} Let $(\CK,\tau)$ be a differentiable Hausdorff geometric field satisfying TIMI. Let $\CM=(M,...)$ be a non-locally modular strongly minimal definable $\CK$-relic whose universe $M$ satisfies $\dim(M)=1$. Then $\CM$ interprets a (strongly minimal) algebraically closed field.
\end{introtheorem}

The proof of Theorem \ref{T: intro field interpretation} represents a streamlined version of the ideas appearing in \cite[\S 7]{CasACF0} and \cite[\S 7]{CaHaYe}. In particular, in this iteration we were able to express many of the relevant notions (such as slopes and tangency) at the level of \textit{types} -- and we hope this has made the geometric ideas more transparent.\\

We conclude this introduction with a remark on background and terminology. Our results build heavily on \cite{CaHaYe} and on \cite{CasOmin}. In most cases, the results quoted (as well as the relevant terminology) can be used as black boxes. We refer the reader to the relevant sections in those papers for a more detailed survey of the history and background of the conjectures addressed below, as well as for a complete overview of terminology. For the readers' benefit, we will include, as we go, definitions of the main tools from  those papers. 

\section{Strongly Minimal Relics of Imaginary Sorts}

In \cite[Theorem 12.26]{CaHaYe}, we proved that every non-locally modular strongly minimal ACVF$_{0,0}$-relic embeds definably into a power of the valued field or the residue field. This allowed us to deduce the trichotomy for \textit{all} strongly minimal ACVF$_{0,0}$ relics as a corollary of the case of \textit{definable} relics. The method appearing in \cite{CaHaYe} is fairly general, and seems useful in other settings. In this section, we give it a general treatment. We then give two applications in the ensuing sections: first, we reduce the trichotomy for strongly minimal relics of \textit{all} o-minimal structures to the simpler case of o-minimal \textit{fields}. Then, analogously to \cite{CaHaYe}, we reduce the trichotomy for strongly minimal relics of real closed valued fields to the definable case.  

\textbf{Throughout this section, fix a sufficiently saturated structure $\mathcal K$.} Our results will assume only a rudimentary integer valued notion of rank on $\mathcal K$. For the purposes of this section, we define:

\begin{definition}\label{D: weakly ranked}
    The structure $\mathcal K$ is \textit{weakly ranked} if there is a rank function $(a,A)\mapsto d(a/A)$, defined on finite tuples over parameter sets in models of  $\theo(\CK)$, satisfying the following properties: 
\begin{enumerate}
    \item Integer valued: $d(a/A)\in \Nn$ for all $a$ and $A$, and $d(a/A)=0$ iff $a$ is algerbaic over $A$. 
    \item Invariant: $d(a/A)=d(a'/A)$ if $\tp(a/A)=\tp(a'/A)$ (thus $d$ is well-defined on complete types), and if $p(x), q(x)$ are complete types such that $q\vdash p$ then $d(q)\le d(p)$. 
    \item Monotonicity: $d(ba/A)=d(ab/A)\ge d(a/A)$ for all $a$, $b$, and $A$.
    \item Sub-additivity: for any $a,b,A$ we have $d(ab/A)\le d(a/A)+d(b/Aa)$, and in an $|A|^+$-saturated model there exists $b'\equiv_A b$ such that $d(ab'/A)=d(a/A)+d(b'/A)$. 
    \item Boundedness:  There exists $k\in \Nn$ such that $d(a)<k$ for all $a\in K$. 
    \item Existence of independent extensions: for all $a$ and $A\subset B$ in a $|B|^+$-saturated model, there is $a'\equiv_Aa$ with $d(a/A)=d(a/B)$.
\end{enumerate}
\end{definition}

 Thus, the class of weakly ranked structures includes all geometric structures, structures of finite burden (in particular structures of finite dp-rank),  and structures of finite  $U^{\text\th}$-rank or SU-rank. Note that Morley rank is \textit{not} in general sub-additive; thus, while structures of finite Morley rank are weakly ranked, this is only witnessed by $U$-rank.

 \begin{remark}
     In any structure $\CK$ we can define $d_{\acl}(\bar a/A)$ to be the minimal length of a tuple $a'\sub \bar a$ such that $\acl(Aa')=\acl(A\bar a)$, and this $d$ satisfies (1), (2), (5), (6), and the first part of (4) from Definition \ref{D: weakly ranked}. In particular, if $d$ also satisfies (3) and the second half of (4), then it defines a weak rank on $\mathcal K$.
 \end{remark}

 Let us stress that the weak rank is defined, a priori, only on tuples from $K$ itself (but not on elements of imaginary sorts). For example, if $\CK$ has finite dp-rank then all imaginary sorts (with their induced structure) are also weakly ranked (because finite dp-rank passes to the imaginary sorts). But if $\CK$ is a geometric structure, there is no obvious way of extending dimension to imaginary sorts (as the usual extension of the geometric dimension to imaginary sorts allows infinite 0-dimensional sorts). For example, if $\CK$ is a 1-h-minimal valued field, then $\CK$ is a geometric structure and therefore weakly ranked, but the structure on the RV-sort may be arbitrary, and thus probably not weakly ranked. \\

\noindent \textbf{For the rest of this section, assume $\mathcal K$ is weakly ranked.}\\ 



Note that we can extend the rank function to definable sets: for any $A$-definable set $X$ (in an $|A|^+$-saturated model), we define $d(X/A):=\max\{d(a/A): a\in X\}$. Indeed, by boundedness, monotonicity, and sub-additivity, the values $d(a/A)$ for $a\in X$ are bounded, so this is well-defined. It then follows from the existence of independent extensions that $d(X/A)$ does not depend on the particular set $A$, so that we may simply write $d(X)$. Then for $A$-definable $X$, and $a\in X$, we say that $a$ is \textit{$A$-generic} in $X$ if $d(a/A)=d(X)$. \\

Recall the following definitions from \cite{CaHaYe}:

\begin{definition}\label{D: locally interalg}
    Let $X$ and $Y$ be interpretable sets.
    \begin{enumerate}
        \item $X$ and $Y$ are in \textit{finite correspondence} if there is a definable $Z\subset X\times Y$ such that both projections of $Z$ are uniformly finite-to-one and surjective.
        \item $X$ is \textit{almost embedable into $Y$} if $X$ is in finite correspondence with a definable subset of $Y$.
        \item $X$ and $Y$ are \emph{locally interalgeraic} if some infinite definable subset of $X$ is almost embeddable into $Y$ (that is, if there are infinite definable $X'\subset X$ and $Y'\subset Y$ such that $X'$ and $Y'$ are in finite correspondence).
    \end{enumerate}
\end{definition}

\begin{definition}
    \begin{enumerate}
        \item For $k\in\mathbb N$, a \textit{$k$-rich configuration} is a tuple $(a,b,A)$ such that (i) $a\notin\acl(Ab)$, (ii) for any infinitely many distinct $b_1,b_2,...\equiv_Ab$ we have $a\in\acl(Ab_1b_2...)$, and (iii) $d(b/A)\geq k$.
        \item A definable set $X$ is \textit{rich} if there is $n$ such that for any $k$, there are $A$, $m$, $a\in D^n$, and $b\in D^m$ such that $(a,b,A)$ is a $k$-rich configuration.
    \end{enumerate}
\end{definition}

Richness is an analog of non-local modularity in the weakly ranked setting. Our goal is to show that if $\mathcal M$ is a non-locally modular strongly minimal $\mathcal K$-relic, and $M$ is locally interalgebraic (in $\mathcal K$) with some definable set $D$, then (1) $D$ is rich and (2) $M$ is almost internal to $D$. The proofs are essentially identical to those in \cite{CaHaYe}. \textbf{Thus, for the rest of this section, we fix a strongly minimal definable $\CK$-relic $\mathcal M$. Passing to an elementary extension and adding parameters if necessary, we may assume that (1) $\CK$ and $\CM$ are both $|\mathcal L(\mathcal M)|^+$-saturated, (2) every $\emptyset$-definable set in $\CK$ is also $\emptyset$-definable in $\mathcal M$, and (3) $\acl_{\CM}(\emptyset)$ is infinite.}

\begin{remark}\label{R: acl inf} Note that by (3) above, every tuple from $\CM^{eq}$ is interalgebraic (in $\CM$) with a tuple of elements of $\CM$. This follows from weak elimination of imaginaries in strongly minimal structures with infinite $\acl(\emptyset)$ -- which in turn follows since every algebraically closed set (being infinite by assumption) is a model.
\end{remark}

Our first observation is that $d$-rank interacts well with Morley Rank:

\begin{lemma}\label{L: d rank vs mr}
    Let $A\subset M$, and let $X\sub M$ be an infinite $\CK(A)$-definable set. Then for any $n$, any  $a\in X^n$ $d$-generic over $A$ is also $\CM$-generic over $A$ in $M^n$. 
\end{lemma}
\begin{proof} Suppose not and fix a counter example $a=(a_1,...,a_n)$. Our assumption implies that there $i$ with $a_i\in\acl_{\CM}(Aa_1...a_{i-1})$. Then $d(a_i/Aa_1...a_{i-1})=0$, and it follows by sub-additivity that $d(a/A)\leq(n-1)d(X)<n\cdot d(X)$ (here we use that $X$ is infinite). On the other hand, by the second clause of sub-additivity, one can show easily that $d(X^n)=n\cdot d(X)$, so by genericity, $d(a/A)=n\cdot d(X)$, a contradiction.
\end{proof}


The following definition is borrowed from \cite{HaHaPe}: 
\begin{definition}
    Keeping the notation of the previous lemma, suppose that $M$ is locally interalgerbaic with a definable set $D$. We say that $X\sub M$ is $D$-\emph{critical} if it has maximal $d$-rank among all definable subsets of $M$ almost embeddable into a power of $D$.
\end{definition}

Using the above terminology, we can replicate (Lemma \ref{L: almost internal technical} below) the main relevant technical result from \cite{CaHaYe} (Proposition 12.13]).

\begin{remark} Below, we remind that a definable family $\CC:=\{C_t:t\in T\}$ of plane curves in $\CM$ is \textit{almost faithful} if for each $t\in T$, there are only finitely many $s\in T$ so that $C_s\cap C_t$ is finite. In this case, the \textit{rank} of the family $\CC$ is the Morley Rank of $T$. Recall that $\CM$ is not locally modular if and only if it defines almost faithful families $\{C_t:t\in T\}$ of plane curves with $T$ of arbitrarily large Morley rank (equivalently, of every Morley rank). See \cite[\S 2]{CasACF0} for an extended treatment of such topics.
\end{remark}

\begin{lemma}\label{L: almost internal technical}
    Let $D$ be a $\CK$-definable set locally interalgebraic with $M$, and let $X\subset M$ be $D$-critical. Let $A\subset\mathcal M^{eq}$ such that $X$, as well as a finite correspondence witnessing it, are definable over $A$. Let $\CC:=\{C_t:t\in T\}$ be an $\CM(A)$-definable almost faithul family of plane curves with $T\sub M^2$ of Morley Rank $2$. Let $u_0\in M^2$ be $\CM(A)$-generic and $x_0\in X^2$ be $d$-generic over $Au_0$. Then
    \begin{enumerate}
        \item The set $L:=\{t\in T: u_0,x_0\in C_t\}$ is finite and non-empty. 
        \item If $t\in L$ then $C_t\cap X^2$ is infinite. 
    \end{enumerate}
\end{lemma}
\begin{proof}
       
    By Lemma \ref{L: d rank vs mr}, $x_0$ and $u_0$ are $\CM$-independent generics in $M^2$. (1) is then straightforward and well-known (one just needs that $\mr(T)=2$). For (2), assume for simplicity that $A=\0$. By (1) we get that $L\sub\acl(u_0,x_0)$. If for the sake of contradiction $|X^2\cap C_t|<\infty$ for some $t\in L$, then $x_0\in \acl(t)$, and combined with (1) above, 
    $t$ is interalgebraic with $x_0$ over $u_0$. Since $x_0\in X^2$, and $X$ is almost embeddable into a power of $D$ (over $A=\emptyset$), there is a tuple $\bar d\in D^n$ interalgebraic with $x_0$. So $x_0$, $t$, and $\bar d$ are all interalgebraic over $u_0$. Finally, note that $\mr(t/u_0)\le 1$ -- so by Remark \ref{R: acl inf}, there is some (singleton) $a\in M$ interalgebraic with $t$ (thus also with $x_0$ and $\bar d$) over $u_0$.
    
    Now since $\bar d$ and $a$ are interalgebraic over $u_0$, and by compactness, there is a $u_0$-definable set $Y$ almost embeddable in a power of $D$ and satisfying $a\in Y\subset M$. Then by the criticality of $X$, $d(Y)\leq d(X)$, and thus $$d(x_0/u_0)=d(a/u_0)\leq d(Y)\leq d(X).$$
    
    On the other hand, by the genericity of $x_0$ over $u_0$ we get $d(x_0/u_0)=2d(X)$. So $2d(X)\leq d(X)$, and thus $d(X)=0$, contradicting the assumption that $M$ is locally interalgerabic with $D$. 
\end{proof}

We now proceed with our two main goals. In \cite[Corollary 12.14]{CaHaYe}, we showed that if $\CN$ is a non-locally modular strongly minimal ACVF-relic locally interalgebraic with a sort $D$, then $D$ is rich (where $d$ is the dp-rank). In our context, the same proof gives: 
\begin{theorem}\label{T: rich}
    Assume $\CM$ is not locally modular. If $D$ is $\CK$-definable and locally interalgebraic with $M$, then $D$ is rich. 
\end{theorem}
\begin{proof}
    We omit many details, as they are identical to \cite{CaHaYe}. Fix a $D$-critical set $X\subset M$, and consider an almost faithful family of plane curves in $\mathcal M$, $\mathcal C=\{C_t:t\in T\}$, where $T$ is generic in a large power of $M$. Without loss of generality, assume all of this data is $\emptyset$-definable (in the relevant structure in each case).
    
    Let $T'$ be the set of $t\in T$ so that $C_t\cap X^2$ is infinite. The content of Lemma \ref{L: almost internal technical}, in this setting, is that every rank 2 subfamily of $\mathcal C$ contains a curve from $T'$. It follows that $d(T)-d(T')$ is bounded, and thus $d(T')$ can be arbitrarily large. One then easily gets a $k$-rich configuration for arbitrarily large $k$, of the form $(a,t,\emptyset)$, where $t\in T'$ and $a\in C_t\cap X^2$ is $d$-generic. This \textit{almost} establishes that $X$ is rich, which (since $X$ almost embeds in a power of $D$) implies that $D$ is rich. The only problem is that $t$ is not a tuple from $X$. This can be fixed by a short trick as in \cite{CaHaYe}. One first chooses $t$ to have all but two coordinates in $X$, say $t=zw$ with $z$ an $X$-tuple and $w\in M^2$. One then chooses two additional points $p,q\in C_t\cap X^2$, adds their left coordinates $p_1,q_1$ as parameters, and replaces $w$ with the right coordinates $p_2,q_2$, obtaining the $k'$-rich configuration $(x,zp_2q_2,p_1q_1)$ with $k'=k-2d(X)$.
\end{proof}

Similarly, in \cite[Corollary 12.18]{CaHaYe} we proved that in the setting of Theorem \ref{T: rich}, $M$ is almost internal to $D$. In the current setting, the same proof gives: 
\begin{theorem}\label{T: almost internality}
    Assume $\CM$ is not locally modular. If $D$ is $\CK$-definable and locally interalgebraic with $M$, then $M$ is almost internal to $D$.
\end{theorem}
\begin{proof}
    Let $X\sub M$ be a $D$-critical set. We may assume that $M$, $X$, $D$, and a finite correspondence witnessing almost embeddability, are all definable over $\0$. We may also assume that there is a $\0$-definable almost faithful family of plane curves in $\mathcal M$, $\{C_t\}_{t\in T}$, with $T\sub M^2$ of Morley Rank 2. 

    It will suffice to show that any $v\in M\setminus \acl(\0)$ is algebraic over $X$ (since by almost embeddability, $X\subset\acl(D))$. Fix such a $v$ and let $u\in X$ be $d$-generic over $v$. Let also $(x,y)\in X^2$ be $d$-generic over $(u,v)$. By Lemma \ref{L: d rank vs mr}, $(u,v)$ is $\mathcal M$-generic in $M^2$. So Lemma \ref{L: almost internal technical} gives some $t\in T$ such that $(u,v),(x,y)\in C_t$, $t\in\acl(uvxy)$, and $C_t\cap X^2$ is infinite. So we can find $(x',y')\in C_t\cap X^2$ with $(x',y')\notin\acl(xy)$. Then $(x,y),(x',y')$ are $\CM$-independent generic elements of $C_t$, which implies that $t\in \acl_{\CM}(xyx'y')$. But $(u,v)\in C_t$ implies that $v\in\acl(tu)$, and thus $v\in \acl(uxyx'y')\subset\acl(X)$, as claimed. 
\end{proof}

In the subsequent sections, we will specialize $\CK$ to o-minimal structures and their T-convex expansions. For these, the following corollary extracts the key content from the above theorems: 
\begin{corollary}\label{C: interacl2internal}
    Let $\CK$ be a sufficiently saturated weakly o-minimal structure, and $\CM$ a strongly minimal non-locally modular $\CK$-relic. Let $D$ be any $\CK$-interpretable set locally interalgerbaic with $M$. Then $D$ is rich (in the sense of dp-rank) and $M$ is internal to $D$. 
\end{corollary}
\begin{proof}
    By saturation, the theory of $\CK$ is weakly o-minimal, and thus dp-minimal. Thus $\CK^{eq}$ is weakly ranked, witnessed by dp-rank. We now view $M$ and $D$ as definable sets in $\CK^{eq}$. By Theorems \ref{T: rich} and \ref{T: almost internality}, $D$ is rich and $M$ is almost internal to $D$. By \cite[Lemma 2.16]{MelRCVFEOI}, in weakly o-minimal structures algerbaic closure coincides with definable closure (including in imaginary sorts). So almost internality to $D$ is the same as internality to $D$. 
\end{proof}

It may also be worth pointing out the somewhat surprising fact that, under the assumptions and notations of Theorem \ref{T: almost internality}, if $D$ is a definable set interalgebraic with some $M'\sub M$ then $M$ is almost internal to any definable infinite subset of $D$ (in particular, $M$ is almost internal to any definable infinite subset of $M$).


\section{Reducing the O-minimal Restricted Trichotomy to Fields}

In this section, we show that the trichotomy for strongly minimal relics of  o-minimal fields implies the same statement over all (dense) o-minimal structures. More precisely, if $\mathcal R$ is o-minimal, and $\mathcal M$ is a non-locally modular strongly minimal $\mathcal R$-relic, we show that there is an interval $I\subset R$ admitting a real closed field structure such that $M$ is internal to $I$. \textbf{Throughout this section, $\CR$ is any sufficiently saturated dense o-minimal structure. We use $\dim$ for dp-rank of tuples from $\CR^{eq}$ (equivalently, the usual (i.e. acl or dcl) dimension extended to imaginary sorts as in \cite{Gagelman} -- this follows by Proposition 3.5 \textit{loc. cit.}).} Thus, $\dim$ defines a weak rank on $\CR^{eq}$, and we view $\CR^{eq}$ as weakly ranked with $\dim$ as the rank function.

Our main tools will be the material from the previous section, in addition to the o-minimal trichotomy theorem of Peterzil and Starchenko \cite{PeStTricho}. Recall that every point in $\CR$ is either trivial, an element of a group interval with semilinear structure, or an element of a real closed field interval. Let us call these \textit{trivial points}, \textit{group points}, and \textit{field points}, respectively. The following is implicit in \cite{PeStTricho}: 

\begin{fact}\label{F: ominimal trichotomy} Let $I\subset R$ be an infinite interval. Then there is an infinite subinterval $J\subset I$ such that one of the following holds:
\begin{enumerate}
    \item Every point in $J$ is trivial.
    \item $J$ has the structure of a semilinear group interval.
    \item $J$ is an expansion of a real closed field.
\end{enumerate}
\end{fact}
\begin{proof} If any element of $I$ is a group point or a field point, we get an interval $J\subset R$ satisfying (2) or (3) (but not necessarily contained in $I$). But (2) and (3) are preserved under subintervals, so we may then shrink $J$ to be contained in $I$ (so we use here that every open interval in a real closed field is in definable bijection with the full field). If no element of $I$ is a group or field point, then we automatically have (1) (with $J=I$).
\end{proof}

The main point of the argument is that cases (1) and (2) in Fact \ref{F: ominimal trichotomy} are incompatible with richness. Let us check this.

\begin{lemma}\label{L: ominimal trivial case} Let $I\subset R$ be an interval consisting only of trivial points. Then $I$ is not rich.
\end{lemma}
\begin{proof}
    The interval $I$ with its full induced structure is o-minimal. By \cite[Theorem 0.5, Theorem 3.2] {MeRuSt} combined with the main result of \cite{PeStTricho}, every definable subset of every $I^n$ is a finite Boolean combination of products of unary and binary relations (in the terminology of \cite{MeRuSt} the structure induced on $I$ is binary). It follows that, after adding a small set of parameters, the acl operator is trivial on $I$ (i.e. for any $A\subset I$, $\acl(A)\cap I$ is the union of $\acl(a)\cap I$ for $a\in A$).

    Now suppose $I$ is rich. Then we can find a $k$-rich configuration $(a,b,A)$, where $A\subset J$, $a\in J^n$, $b\in J^m$, and $k$ is arbitrarily large while $n$ is bounded. In particular, we may fix such a configuration $(a,b,A)$ with $k>n$. Then $\dim(b/A)>\dim(a/A)$, and so $\dim(b/Aa)>0$. It follows that there are infinitely many distinct automorphism conjugates of $b$ over $Aa$, say $b_1,b_2,...$. Let $a_1$ be the first coordinate of $a$. By assumption, $a\in\acl(Ab_1b_2...)$. By triviality, there is $i$ so that $a_1\in\acl(Ab_i)$. Since the $b_i$ are conjugate to $b$ over $Aa$, we get $a_1\in\acl(Ab)$. The same argument applies for all coordinates of $a$, so we conclude that $a\in\acl(Ab)$, a contradiction.
\end{proof}

\begin{lemma}\label{L: ominimal modular case} Let $I\subset R$ be a group interval with semilinear structure. Then $I$ is not rich.
\end{lemma}
\begin{proof} This is proven for semilinear groups in \cite{CaHaYe}, and the proof goes through for group intervals. Namely, suppose $I$ is rich. Then, as in the proof of Lemma \ref{L: ominimal trivial case}, we can find a $k$-rich configuration $(a,b,A)$ where $\dim(b/Aa)>0$. Say that $a\in I^m$ and $b\in I^n$. By compactness, there are $A$-definable sets $T\subset I^n$ and $X\subset I^m\times T$ (viewed as a family of sets $X_t\in I^m$) such that (i) $b$ is $A$-generic in $T$, $a$ is $Ab$-generic in $X_b$, (iii) $(a,b)$ is $A$-generic in $X$, and (iv) the intersection of any infinitely many distinct fibers $X_t$ is finite.

By semilinearity, there is a neighborhood $U=U_a\times U_b$ of $(a,b)$ such that $X\cap U$ is an open subset of a coset of a local subgroup. Thus the sets $X_t\cap U_a$, for $t\in U_b$, are all open subgroups of cosets of a fixed local subgroup $H\leq I^m$. Since $a\notin\acl(Ab)$, $H$ is infinite, which implies that $a$ is not isolated in $a+H$.

Finally, since $\dim(b/Aa)>0$, there are distinct $b_1,b_2,...\in U_b$ with each $X_{b_i}$ containing $a$. Then each $X_{b_i}$ contains a neighborhood of $a$ in $H+a$. Since $a$ is not isolated in $H+a$, it follows by compactness that $\bigcap_iX_{b_i}$ is infinite, contradicting (iv) above.
\end{proof}

Finally, we also need the following (\cite[Lemma 5.1]{HaHaPe}):

\begin{fact}\label{F: locally algebraic one variable}
    Let $\mathcal N$ be any 1-sorted structure, and let $X$ be an infinite interpretable set in $\mathcal N$. Then $X$ is locally interalgebraic with an infinite set of the form $Y/E$, where $Y\subset N$ is definable and $E\subset Y^2$ is a definable equivalence relation. 
\end{fact}

Returning to our o-minimal $\CR$, we can now show:

\begin{theorem}\label{T: reduction2fields} Let $\mathcal M=(M,...)$ be a non-locally modular strongly minimal $\mathcal R$-relic. Then $M$ is internal to an interval in $R$ supporting an $\mathcal R$-definable real closed field structure.
\end{theorem}
\begin{proof} We begin with:

\begin{lemma} There is an infinite interval in $R$ which almost embeds into $M$.
\end{lemma}
\begin{proof} By Fact \ref{F: locally algebraic one variable}, $M$ almost embeds a set $Y/E$, where $Y\subset R$ is definable and $E$ is a definable equivalence relation. Since $Y/E$ is infinite, by o-minimality there must be infinitely many finite classes. Replacing $Y$ by the union of these classes, we may assume all classes are finite. Now replace each class by its smallest element to obtain a definable bijection between $Y$ and an infinite subset of $R$. By o-minimality, this infinite subset of $R$ contains an infinite interval.
\end{proof}

Fix $I$, an infinite interval almost embeddable in $M$. Shrinking $I$ if necessary, by Fact \ref{F: ominimal trichotomy} we may assume $I$ satisfies either (1), (2), or (3) in the statement of Fact \ref{F: ominimal trichotomy}. By Corollary \ref{C: interacl2internal}, $I$ is rich, so option (3) is forced by Lemmas \ref{L: ominimal trivial case} and \ref{L: ominimal modular case}. Thus $I$ supports an $\mathcal R$-definable real closed field structure. Finally, also by Corollary \ref{C: interacl2internal}, $M$ is internal to $I$. 
\end{proof}

\begin{corollary}
    If the Zilber trichotomy holds for strongly minimal relics of o-minimal fields, then the Zilber trichotomy holds for strongly minimal relics of all dense o-minimal structures.
\end{corollary}

\section{Interpretable RCVF-relics}
In this section $\mathcal R_V=(R,+,\cdot,V)$ is a sufficiently saturated model of the theory of real closed valued fields (so $\CR=(R,+,\cdot)\models RCF$ and $V$ is a convex valuation; see the remark at the end of the section for a discussion of generalizations). We denote the valuation ring, value group, and residue field by $V$, $\Gamma$, and $\textbf k$, respectively. We will show that every non-locally modular strongly minimal $\CR_V$-relic definably embeds into a power of either $R$ or $\textbf k$. This reduces the study of non-locally modular strongly minimal $\CR_V$-relics to the definable case. Indeed, since $\textbf{k}$ is a pure real closed field, it is elementarily equivalent to a definable $\mathcal R_V$-relic -- and thus so is every definable $\textbf{k}$-relic.

Our strategy is similar to \cite[\S 12]{CaHaYe}, building heavily on ideas from \cite{HaHaPe}: we note that the universe of the relic is locally interalgebraic with one of four distinguished sorts (including $R$ and $\textbf k$); rule out the other two using the richness clause of Corollary \ref{C: interacl2internal}; and then produce the required embedding using the other clause of Corollary \ref{C: interacl2internal}.

\subsection{On Elimination of Certain Imaginaries} Before proceeding, we need to adapt a result from \cite{HaOnPi} on elimination of certain imaginaries in ACVF. Recall that the dimension function in a geometric structure can be extended to imaginary sorts in a way that preserves additivity, but may introduce infinite sets of dimension 0 (\cite{Gagelman}, see also \cite{JohnQpTop} for a brief self contained exposition). Let us say that a geometric structure $\mathcal K$ has \textit{elimination of imaginaries down to 0-dimensional sorts} if every interpretable set embeds definably (over parameters) into one of the form $K^n\times S$ with $\dim(S)=0$. The well-known elimination of imaginaries in ACVF down to the \textit{geometric sorts} is an example, because the geometric sorts are all 0-dimensional. Moreover, it was shown in \cite{MelRCVFEOI} that RCVF also eliminates imaginaries down to the geometric sorts, and thus down to 0-dimensional sorts.

In \cite[Lemma 2.8]{HaOnPi}, it is shown that if $X$ is interpretable in ACVF and not locally interalgebraic with the value group, the residue field, or the set of closed balls of a given radius, then $X$ embeds into a power of the valued field. The proof works more generally for interpretable sets with no infinite 0-dimensional subsets. We now briefly review and expand on the argument in order to give it a more general setting. There are three main steps. We begin with:

\begin{lemma}\label{L: geometric weak e of i} Let $\mathcal K$ be a geometric structure with elimination of imaginaries down to 0-dimensional sorts. Let $X$ be a $\mathcal K$-interpretable set with no infinite 0-dimensional interpretable subsets. Then $X$ is in finite correspondence with a definable set.
\end{lemma}
\begin{proof} By elimination of imaginaries down to 0-dimensional sorts, we may assume $X\subset K^n\times S$ where $S$ is interpretable with $\dim(S)=0$. Then the fibers of the projection $\pi:X\rightarrow K^n$ are 0-dimensional subsets of $X$, and by assumption they are all finite. This gives a finite correspondence between $X$ and $\pi(X)\subset K^n$.
\end{proof}

For the next step, let us say that a sufficiently saturated 1-sorted structure $\mathcal K$ has the \textit{acl-model property} if for some small set $A\subset K$, every algebraically closed set $B$ with $A\subset B\subset K$ is an elementary submodel of $\mathcal K$. It is well-known that ACVF and RCVF have the acl-model property (where $A$ is a single point with non-zero value). In fact, any structure with definable Skolem functions has this property. 

\begin{lemma}\label{L: use of acl model} Let $\mathcal K$ be a sufficiently saturated 1-sorted structure with the acl-model property. Let $X$ be an interpretable set in finite correspondence with a definable set. Then $X$ is a `finite quotient' -- there are a definable set $Y$ and a surjective finite-to-one interpretable map $Y\rightarrow X$.
\end{lemma}
\begin{proof}
Since $X$ is interpretable, there are a definable set $Z$ and an interpretable surjection $f:Z\rightarrow X$. Absorbing parameters, we may assume that $X$, $Z$, and $f$ are $\emptyset$-interpretable. We may also assume the parameter set $A$ witnessing the acl-model property is $\emptyset$; and that $X$ is in $\emptyset$-definable finite correspondence with a definable set -- thus every element of $X$ is interalgebraic with a real tuple.

We claim that for every $a\in X$, there is $b\in f^{-1}(a)$ with $b\in\acl(a)$. Thus, by compactness, there is a $\emptyset$-definable subset $Y\subset Z$ such that the restriction $f\restriction_Y:Y\rightarrow X$ is finite-to-one and surjective, completing the proof.

So let $a\in X$; we find such a tuple $b$. Let $\mathcal L$ be the 1-sorted language of $\mathcal K$, and $\mathcal L^{eq}$ its expansion by interpretable sorts. Let $c$ be a real tuple interalgebraic with $a$. Let $\phi(a,c)$ be an $\mathcal L^{eq}$-formula isolating $\tp(a/c)$. Let $\theta(y,z)$ be the $\mathcal L^{eq}$-formula $\phi(f(y),z)$. Then $\theta$ has only real tuples as free variables, so $\theta$ is equivalent to an $\mathcal L$-formula $\psi(y,z)$. Now the surjectivity of $f$ implies that $\models\exists y\theta(y,c)$, and thus $\models\exists y\psi(y,c)$. By the acl-model property, there is $b'\in\acl(c)=\acl(a)$ so that $\models\psi(b',c)$, and thus $\models\theta(b',c)$. Let $a'=f(b')$. Then $\models\phi(a',c)$, so $\tp(a'c)=\tp(ac)$, and thus there is $b$ with $\tp(a'b'c)=\tp(abc)$. So $b\in\acl(a)$ and $f(b)=a$, as desired.
\end{proof}

Finally, the third step is well-known:

\begin{lemma}\label{L: e of fin i} Let $\mathcal K$ be an expansion of a field. Let $X$ be a `finite quotient' -- i.e. a finite-to-one interpretable image of a definable set. Then $X$ embeds definably into a power of $K$.
\end{lemma}
\begin{proof} This is just elimination of finite imaginaries in fields (using symmetric functions).
\end{proof}

Finally, putting the last three lemmas together, we obtain:

\begin{proposition}\label{P: JP}
    Let $\mathcal K$ be a sufficiently saturated geometric expansion of a field with the acl-model property and elimination of imaginaries down to 0-dimensional sorts (for example, any sufficiently saturated model of ACVF or RCVF, or any sufficiently saturated $p$-adically closed field). Let $X$ be a $\mathcal K$-interpretable set with no infinite 0-dimensional subsets. Then $X$ embeds definably into a power of $K$.
\end{proposition}

\begin{remark}\label{R: no EOI}
    In \cite[Theorem 1.2]{HaHrMac3} it is shown that $T$-convex expansions of RCVF need not admit elimination of imaginaries down to the standard geometric sorts (of \cite{MelRCVFEOI}). As far as we could ascertain, the only property of the geometric sorts used in the proof is that for any geometric sort $S$, any interpretable function $f: S\to K^n$ has finite image. Since $K^n$ has no infinite $0$-dimensional definable subsets, and for any interpretable function with domain $U$, $\dim(U)\ge \dim(f(U))$, the same applies if $S$ is any $0$-dimensional sort. Thus, it seems that the proof of \cite{HaHrMac3} implies that $T$-convex expansions of RCVF need not eliminate imaginaries down to $0$-dimensional sorts.  
\end{remark}

\subsection{Eliminating Interpretable Relics}

Returning to our structure $\CR_V\models RCVF$, we now deduce the main result of this section:

\begin{theorem}\label{T: RCVF}
    Let $\mathcal M=(M,...)$ be a non-locally modular strongly minimal $\mathcal R_V$-relic. Then $M$ embeds $\mathcal R_V$-definably into a power of either $R$ or $\textbf k$.
\end{theorem}

\begin{proof} Throughout, we use that $\mathcal R_V$ is weakly ranked, with the rank function $d$ given by dp-rank (indeed, this follows since RCVF is dp-minimal). \\

Our main tool is a result of Halevi, Hasson, and Peterzil (\cite[Proposition 5.5]{HaHaPe}), stating that every infinite $\CR_V$-interpretable set is locally interalgebraic with at least one of $R$, $\Gamma$, $\textbf{k}$, and $R/V$ (the quotient of the group $(R,+)$ by the subgroup $V$). In particular, $M$ is locally interalgebraic with one of these four sorts. Our first goal is to rule out two of the four:

\begin{claim}\label{Cl: ruling out linear sorts} $M$ is not locally interalgebraic with $\Gamma$ or $R/V$.
\end{claim}
\begin{proof} Let $D$ be either $\Gamma$ or $R/V$. If $M$ is locally interalgebraic with $D$ then by Theorem \ref{T: rich}, $D$ is rich. On the other hand, note that $D$ is a \textit{locally linear group} -- every definable subset of $D^n$ locally coincides with a coset in a neighborhood of a generic point (for $\Gamma$ see \cite[Theorem B]{vdDries-Tconvex}, and for $R/V$ see \cite[Proposition 6.16]{HaHaPe}, using the fact that $\CR_V$ has definable Skolem functions \cite[Remark 2.4]{vdDries-Tconvex}). By local linearity, it follows that $D$ is not rich (the proof from \cite[Proposition 12.24]{CaHaYe} works word-for-word here), a contradiction.
\end{proof}

So $M$ is locally interalgebraic with either $R$ or $\textbf k$. First assume $M$ is locally interalgebraic with $\textbf k$. Then by Corollary \ref{C: interacl2internal}, $M$ is internal to $\textbf k$. Since $\textbf k\models RCF$, it eliminates imaginaries, and so $M$ embeds into a power of $\textbf k$.

Finally, assume $M$ is not locally interalgebraic with $\textbf k$. Then, out of the four distinguished sorts $R,\Gamma,R/V,\textbf k$, $M$ is only locally interalgebraic with $R$ (we say that $M$ is \textit{$R$-pure}). It follows that $M$ has no infinite 0-dimensional interpretable subsets (otherwise $M$ would have to be locally interalgebbraic with one of the other three distinguished sorts). Since RCVF has elimination of imaginaries down to 0-dimensional sorts (\cite{MelRCVFEOI}) and definable Skolem functions (\cite[Remark 2.4]{vdDries-Tconvex}), we can conclude by Proposition \ref{P: JP}. 
\end{proof}

\begin{remark}\label{R: imaginaries} 

As pointed out in Remark \ref{R: no EOI}, the proof of Theorem \ref{T: RCVF} does not extend to general $T$-convex expansions of o-minimal fields. Instead, we propose a different method for handling the imaginary sorts in (power bounded) $T$-convex fields. The method is not worked out, and is too involved for the scope of this paper, but we briefly summarize. In the power bounded case, $\Gamma$ and $R/V$ are known to be locally linear (this is not known in the exponential case). So assuming power boundedness, it follows as above that $M$ is not internal to $\Gamma$ or $R/V$. Now consider the other two cases (that $M$ is internal to $\textbf k$, or that $M$ is $R$-pure).

By (\cite[Theorem B]{vdDries-Tconvex}), $\textbf k$ is a model of the underlying o-minimal theory of $R$ -- so if $M$ is internal to $\textbf k$, then elimination of imaginaries identifies $M$ with a definable $\textbf k$-relic, and thus (up to elementary equivalence) with a definable $\mathcal R$-relic.

If $M$ is $R$-pure, then as above, every 0-dimensional subset of $M$ is finite. In particular, the geometric dimension gives an additive weak rank on $M$. One can then try to mimic (\cite[Theorem 3.4]{JohnEOI}, following the path of \cite[Theorem 4.29]{JohnQpTop} ) to show that $M$ carries a definable `generic manifold' topology (i.e. $M$ is almost everywhere locally homeomorphic to open subsets of $R^n$ for some $n$). If this succeeds, one can rerun the relevant parts of the papers \cite{CaHaYe} and \cite{CasOmin}, replacing `geometric structures' with `structures admitting an additive weak rank' (thus instead of formally reducing to definable relics, we are suggesting to generalize the proof from the definable case). The idea is to use the manifold structure on $M$ to `generically and locally' inherit any necessary geometric input from $R$ (tangent spaces, purity of ramification, etc.) As the methods of \cite{CaHaYe} and \cite{CasOmin} almost everywhere use only generic and local data, we expect this approach to suffice.  This is similar in spirit to the next two sections of this paper, which exploit the underlying o-minimal structure of a T-convex field as a \textit{local reduct} in proving purity of ramification.
\end{remark}

\section{T-Convex Fields as Hausdorff Geometric Structures with a Lore}

\textbf{For the rest of the paper, we fix $T$, a complete o-minimal expansion of the theory of real closed fields. We let $T_{conv}$ be the expansion of $T$ considered in \cite{vdDries-TconvexCorrect}, by naming a convex subring closed under all  continuous $\emptyset$-definable total functions. We fix a sufficiently saturated model $\mathcal R_V\models T_{conv}$ with underlying model $\mathcal R\models T$ and distinguished convex subring $V\subset R$. Note that we do not assume $\mathcal R$ is power pounded.} The rest of the paper will focus on strongly minimal definable $\CR_V$-relics.

In \cite{CasOmin}, the `higher dimensional' case of the trichotomy was proved for relics of o-minimal fields (by the `higher dimensional case', we mean the statement that strongly minimal relics on universes of dimension at least 3 are locally modular). The proof was done in the language of an abstract \textit{local reduct} (or lore, for shorthand) of a \textit{t-minimal Hausdorff geometric structure}. Recall (see \cite[Definition 2.8]{CasOmin}):
\begin{definition}\label{D: HGS}
    Let $\CK$ be a sufficiently saturated geometric structure, and $\tau$ a Hausdorff topology on $K$ with a uniformly $\emptyset$-definable basis (which we extend to all $K^n$ via the product topology). Then $(\CK,\tau)$ is a \textit{t-minimal Hausdorff geometric structure} if the following hold:
    \begin{enumerate}
        \item $K$ has no isolated points.
        \item If $\emptyset\neq X\subset K^n$ is definable, and $\operatorname{Fr}(X)$ is the frontier of $X$ (i.e. $\overline X-X$), then $\dim(\operatorname{Fr}(X))<\dim(X)$.
        \item If $\emptyset\neq X\subset K^n$ and $f:X\rightarrow K$ are definable, then the there is a definable $X'\subset X$ with $\dim(X-X')<\dim(X)$ such that $f$ is continuous on $X'$.
    \end{enumerate}
\end{definition}
\begin{remark} The general definition of a Hausdorff geometric structure is more complicated -- see \cite[Definition 3.1]{CaHaYe}. However, the two definitions are equivalent for a structure with a $\emptyset$-definable topology. In this paper, we will only consider $\emptyset$-definable topologies, so the t-minimal clause is automatic throughout. In particular, we will generally refrain from referring to t-minimality when citing \cite{CasOmin}.
\end{remark}

Suppose $(\CK,\tau)$ is a t-minimal Hausdorff geometric structure. Roughly speaking, a \textit{lore} of $(\CK,\tau)$ is a `rich enough' collection of definable sets satisfying certain closure properties in addition to a key `locality' property. The idea was that in the course of the proof in that paper, one only needs the sets in a given lore (called `loric sets'); so if certain geometric properties only hold for sets in a given lore, no harm is done.
As was mentioned in \cite{CasOmin}, the key reason for using lores was to later generalize to RCVF, where crucial features of the proof work only in the lore of `locally o-minimal' sets. 

In this section, we make everything in the previous paragraph precise. That is, we endow $\CR_V$ with the data of a differentiable\footnote{A Hausdorff geometric structure is \textit{differentiable} if it comes equipped with suitable notions smoothness and tangent spaces -- where tangent spaces are taken at smooth points of definable sets. See Definitions 3.27 and 3.34 of \cite{CaHaYe}.} Hausdorff geometric structure with a fixed lore (consisting of sets which are locally $\CR$-definable). In the next sections, we use this data to adapt the main theorem of \cite{CasOmin} to $\CR_V$.

We need the following properties of $\CR_V$: 
\begin{fact}\label{F: basic Tconv facts}
    \begin{enumerate}
        \item $\CR_V$ is weakly o-minimal with the exchange property (thus $\CR_V$ is geometric). In particular $\CR_V$ is $t$-minimal and, in fact, visceral (since the topology is uniform). See \cite{JohnVisc} for details and definitions.  
        \item (\cite[Theorem 3.10]{vdDriesLewen}) $T_{conv}$ admits quantifier elimination relative to $T$. Specifically, if $T$ has quantifier elimination in some expanded language then $T_{conv}$ has quantifier elimination in the same language expanded by a symbol for the valuation ring and function symbols for all total $\CR$($\0$)-definable functions. 
    \end{enumerate}
\end{fact}

Fact \ref{F: basic Tconv facts}(1) implies various nice geometric properties (see \cite{JohnVisc}). For example, one gets the frontier inequality for $\CR_V$-definable sets, as well as the fact that the local dimension of a definable set $X$ at any generic point is $\dim(X)$. 

Quantifier elimination also implies the following:

\begin{lemma} Let $X\subset R^n$ be $\CR_V(A)$-definable. Then there are $\CR(A)$-definable sets $X_1,...,X_k\subset R^n$, and $\CR_V(A)$-definable open subsets $U_i\subset X_i$, such that $X=\bigcup_{i=1}^kU_i$.
\end{lemma}
\begin{proof} By quantifier elimination, $X$ is defined over $A$ by a quantifier free formula in the language consisting of all $\CR(\emptyset)$-definable sets and functions, and a predicate for $V$. Since the desired form of $X$ is preserved under unions, it suffices to consider a conjunction of atomic and negated atomic formulas. Every atomic formula occurs in the language of $\CR$, except those of the form $f(\bar x)\in V$ for an $\CR(\emptyset)$-definable function $f$. 

So it suffices to show that if $Y$ and $f:Y\rightarrow R$ are $\CR(A)$-definable, then the formulas $f(\overline y)\in V$ and $f(\overline y)\notin V$ have the desired form. We argue this by induction on $\dim(Y)$. For the inductive step, let $Y'\subset Y$ be the set of points where $f$ is continuous; then $Y'$ is $\CR(A)$-definable and satisfies $\dim(Y-Y')<\dim(Y)$. Now since $f$ is continuous on $Y'$, the sets $f(\overline y)\in V$ and $f(\overline y)\notin V$ restrict to clopen subsets of $Y'$. Combining this with an inductive treatment of $Y-Y'$ finishes the proof.
\end{proof}

Following \cite{CaHaYe}, let us say that a definable set $X'$ is a \textit{d-approximation} of a definable set $X$ at a point $a\in X$ if (1) $\dim(X')=\dim(X)$ and (2) $X$ and $X'$ agree in a neighborhood of $a$. In this language, we conclude:

\begin{corollary}\label{C: generically locally ominimal}
    Let $X\subset R^n$ be $\CR_V(A)$-definable, and $a\in X$ generic over $A$. Then $X$ has an $\CR(A)$-definable d-approximation at $a$.
\end{corollary}
\begin{proof}
 The above lemma gives us an $\CR(A)$-definable $X_i$, and an open $\CR_V(A)$-definable $U_i\subset X_i$ containing $a$ and contained in $X$. Let $d_a$ be the local dimension of $U_i$ at $a$. Replacing $X_i$ if needed by the (definable) set of points of $X_i$ of local dimension at most $d_a$, we may assume that $\dim(X_i)\leq d_a\leq\dim(X)$. But then since $a$ is generic in $X$ over $A$, it follows easily that $\dim(U_i)=\dim(X_i)=\dim(X)$, and $a$ is generic in each of these three sets over $A$. By the frontier inequality, each of $X$, $X_i$, and $U_i$ has the same germ at $a$. Thus $X_i$ is an $\CR(A)$-definable d-approximation.
\end{proof}

It was proved in \cite[\S 9.1]{CaHaYe} that o-minimal expansions of fields are \textit{differentiable Hausdorff geometric structures} (see \cite[\S ]{CaHaYe} for the relevant definitions). The proof covered 1-h-minimal theories, so in particular power bounded $T$-convex theories. Using Corollary \ref{C: generically locally ominimal}, the same result easily expands to $\CR_V$. 

\begin{proposition}
    $\CR_V$ is a differentiable Hausdorff geometric structure. 
\end{proposition}
\begin{proof}
    It was noted in \cite{CaHaYe} (using results from \cite{JohnVisc}) that visceral structures with the exchange property are Hausdorff geometric structures (recall that \textit{visceral} structures are $t$-minimal structures whose topology is uniform, see \cite{DolGodViscerality} or \cite{JohnVisc} for details). So we only need to show $\CR_V$ is differentiable. Recall this requires assigning every definable set to a \textit{smooth locus}, and assigning every point of the smooth locus to a \textit{tangent space} satisfying certain properties (see \cite[Definition 3.34]{CaHaYe}).
    
    Since smoothness and tangent spaces are $d$-local (i.e. invariant under d-approximations), they are automatically inherited from $\CR$ to $\CR_V$. Namely, let $X\subset R^n$ be $\CR_V$-definable, and $a\in X$. Say that $a$ is a smooth point in $X$ if (1) there is an $\CR$-definable d-approximation of $X$ at $a$, and (2) for some (equivalently any) such d-approximation $X'$, $a$ is smooth in $X'$ in the sense of $\CR$. Moreover, in this case, define $T_x(X)$ as exactly $T_x(X')$ for some (equivalently any) d-approximation $X'$. It follows easily that this data turns $\CR_V$ into a differentiable Hausdorff geometric structure. Indeed, the only potential difficulty is showing that generic points of definable sets are smooth -- and this follows from the analogous statement in $\CR$ and Corollary \ref{C: generically locally ominimal}.
\end{proof}

Our next goal is to endow $\CR_V$ with an appropriate \textit{lore}, in the sense of \cite[Definition 2.8]{CasOmin}:

\begin{definition}\label{D: lore}
    A \textit{lore} of $\CR_V$ is a collection $\CS$ of $\mathcal R_V$-definable sets (called \textit{loric sets}) such that:
    \begin{enumerate}
        \item\label{D: loric sets basic} (Basic Properties) $R$ is loric. The diagonal $\{(x,x)\}\subset R^2$ is loric. The loric sets are closed under finite products, finite intersections, and coordinate permutations.
        \item\label{D: loric sets invariance} (Invariance) If $f:X\rightarrow Y$ is an $\CR_V$-definable homeomorphism, $X$ is loric, and the graph of $f$ is loric, then $Y$ is loric.
        \item\label{D: loric sets loc} (Locality) Every $\CR_V$-definable open subset of a loric set is loric. If $X$ is $\CR_V$-definable, and $\{U_i:i\in I\}$ is an open cover of $X$ by loric sets, then $X$ is loric.
        \item\label{D: loric sets gen} (Genericity) If $X\neq\emptyset$ is $\CR_V(A)$-definable, then there is a relatively open $\CR_V(A)$-definable $X'\subset X$ such that $\dim(X-X')<\dim(X)$ and $X'$ is loric.
    \end{enumerate}
\end{definition}


Precisely, we will take the \textit{locally $\CR$-definable sets}:

\begin{definition} If $X\subset R^n$ is $\CR_V$-definable, we call $X$ \textit{locally $\CR$-definable} if $X$ admits an open cover by $\CR$-definable sets (note that we do not require the open cover to be definable as a family).
\end{definition}

Now we show:

\begin{proposition}\label{P: lore}
    Let $\CS$ be the class of $\CR_V$-definable sets that are locally $\CR$-definable. Then $\CS$ is a lore. 
\end{proposition}
\begin{proof}
    Items (1)-(3) of Definition \ref{D: lore} are an easy exercise. We show the genericity property. So let $X\neq\emptyset$ be $\CR_V(A)$-definable. By Corollary \ref{C: generically locally ominimal}, $X$ admits an $\CR(A)$-definable d-approximation at any generic point over $A$. By compactness, there is a single formula $\phi(x,y)$ over $A$ (in the language of $\CR$) witnessing this -- i.e. for every generic $a\in X$ over $A$, some instance $\phi(x,b)$ defines a d-approximation of $X$ at $a$. Now let $X'$ be the set of all $a\in X$ such that some instance $\phi(x,b)$ defines a d-approximation of $X$ at $a$. Then $X'$ is $\CR_V(A)$-definable, open in $X$, locally $\CR$-definable, and contains all generics of $X$ over $A$, which is enough. 

\end{proof}

\section{Purity of Ramification}

The proof of the `higher-dimensional' case of the trichotomy in \cite{CasOmin} crucially used a \textit{purity of ramification} statement for o-minimal structures. Before stating this result, we recall that a \textit{loric map} is a continuous map with loric graph, and a \textit{loric $n$-manifold} is a definable set everywhere locally homeomorphic to open subsets of $n$-space, via homeomorphisms that are also loric maps. 

The following definitions are taken from \cite{CasOmin}:

\begin{definition}
    Let $f:X\rightarrow Y$ be a map of topological spaces. The \textit{ramification locus of $f$}, denoted $\operatorname{Ram}(f)$, is the set of $x\in X$ such that $f$ is not injective on any neighborhood of $x$.
\end{definition}

\begin{definition}
    Let $(\mathcal K,\tau,\mathcal S)$ be a t-minimal Hausdorff geometric structure with a fixed lore. Let $d$ be a positive integer. We say that $(\mathcal R,\tau,\mathcal S,d)$ has \textit{manifold ramification purity of order $d$} if whenever $f:X\rightarrow Y$ is a finite-to-one loric map of loric $n$-manifolds (for the same $n$), then either $\operatorname{Ram}(f)=\emptyset$ or $\dim(\operatorname{Ram}(f))\geq n-d$.  
\end{definition}

In \cite{CasOmin} we proved that o-minimal fields have manifold ramification purity of order 2 (with the `full' lore of all definable sets):

\begin{fact}\label{F: ominimal purity}
    $\mathcal R$ has manifold ramification purity of order 2 when equipped with the full lore.
\end{fact}

We expected the same to hold in $\CR_V$, but this turned out to fail (for RCVF, thus for any choice of $\CR_V$). Let us sketch a counterexample:

\begin{proposition}\label{P: purity failure} $\CR_V$, equipped with the full lore, does not have manifold ramification purity of order $d$ for any $d$.
\end{proposition}
\begin{proof}
    Fix any $n>1$. We write $\vec 0$ for $(0,...,0)\in R^n$. We will exploit the fact that $R^n-\{\vec 0\}$ is not definably connected. Namely, let $A$ be the set of $(x_1,...,x_n)\in R^n-\{\vec 0\}$ such that $$x_n^2=a(x_1^2+...+x_{n-1}^2)$$ for some $a\in V$. Let $B=R^n-\{\vec 0\}-A$. It is easy to see that $A\cup B$ is a definable partition of $R^n-\{\vec 0\}$ into two open sets, each with $\vec 0$ in its closure.
    
    To prove the proposition, we will build a continuous definable map (i.e. a loric map in the full lore) $f:R^n\rightarrow R^n$ with $f(\vec 0)=\vec 0$, whose only ramification point is $\vec 0$. Then $\operatorname{Ram}(f)$ is non-empty of codimension $n$, and since $n$ is arbitrary, the proposition follows.

    To build the map, the idea is to take the union of two continuous injections, one defined on $A$ and one on $B$. Then the union is automatically continuous and locally injective on $A\cup B$, i.e. on $R^n-\{\vec 0\}$. If on the other hand we ensure that the injections both have limit $\vec 0$ at $\vec 0$, and that their images `overlap' arbitarily close to 0, we will be able to witness continuity and ramification at $\vec 0$.
    
    Let us proceed. First, fix $\epsilon$ with $v(\epsilon)>0$. Let $g:A\rightarrow R^n$ be the inclusion, and let $h:B\rightarrow R^n$ be the map $(x_1,...,x_n)\mapsto(x_1,...,\epsilon x_n)$ (that is, scaling only the last coordinate by $\epsilon$).

    \begin{claim} Let $U\subset R^n$ be any neighborhood of $\vec 0$. Then there is $(a,b)\in U^2\cap(A\times B)$ with $h(b)=a$.
    \end{claim}
    \begin{proof}
        Without loss of generality $U$ is the $n$-dimensional closed $\gamma$-ball for some $\gamma\in\Gamma$ -- i.e. the set of $(x_1,...,x_n)$ with each $v(x_i)\geq\gamma$. Fix $x\in R$ with $v(x)=\gamma+v(\epsilon)$. Then one can check that $b=(x,...,x,\frac{x}{\epsilon})\in B\cap U$, and $$h(b)=a=(x,...,x)\in A\cap U.$$ 
    \end{proof}

    Now define $f:R^n\rightarrow R^n$ by $f(\vec 0)=\vec 0$, $f\restriction_A=g$, and $f\restriction_B=h$. Clearly, $f$ is definable. Next, note that $g$ and $h$ are continuous injections; so since $A$ and $B$ form an open cover of $R^n-\{0\}$, it follows that $f$ is continuous and locally injective on $R^n-\{0\}$. Moreover, one checks easily that $g$ and $h$ both have limit $\vec 0$ at $\vec 0$, which implies that $f$ is also continuous at $\vec 0$ (so is continuous everywhere, and thus loric).
    
    So $f$ is a loric map from the loric $n$-manifold $R^n$ to itself. Moreover, since $g$ and $h$ are injective, clearly $f$ is finite-to-one (in fact at most 2-to-1 everywhere). So $f$ satisfies the hypotheses of manifold ramification purity.

    It remains only to show that $f$ ramifies at $\vec 0$. For this, let $U$ be any neighborhood of $\vec 0$. By the claim, there are $a,b\in U$ with $a\in A$, $b\in B$, and $h(b)=a$. Then $a\neq b$, and $f(a)=f(b)=a$, witnessing ramification.   
\end{proof}

Note that the construction in Proposition \ref{P: purity failure} crucially uses the abundance of clopen sets provided by the valuation ring. In particular, it is easy to see that the map $f$ from the proof cannot be RCF-definable in any neighborhood of $\vec 0$ (if it were, then from the equation $f(a)=a$ one could recover the restriction of $A$ to a neighborhood of $\vec 0$, and from there it is straightforward to recover $V$).

This counterexample (and the fact that it is not locally $\CR$-definable) was the original motivation for introducing lores. In fact, we now rather immediately show:

\begin{proposition}\label{P: purity success}
     $\CR_V$ does satisfy manifold ramification purity of order 2 when equipped with the lore of locally $\CR$-definable sets.
\end{proposition}
\begin{proof} Let $f:X\rightarrow Y$ be a finite-to-one loric map of loric $n$-manifolds. Suppose $\operatorname{Ram}(f)$ is non-empty, and let $a\in\operatorname{Ram}(f)$. Since $X$, $Y$, and $f$ are loric, one can choose open neighborhoods $X'\subset X$ of $a$ and $Y'\subset Y$ of $f(a)$ so that $f(X')\subset Y'$ and $X'$, $Y'$, and $f\restriction_{X'}$ are $\CR$-definable. Note that $f\restriction_{X'}$ also ramifies at $a$, and that $X'$ and $Y'$ are definable $n$-manifolds in the o-minimal sense. Then by Fact \ref{F: ominimal purity}, we conclude that $$\dim(\operatorname{Ram}(f\restriction_{X'}))\geq n-2.$$ But clearly $$\operatorname{Ram}(f\restriction_{X'})\subset\operatorname{Ram}(f),$$ so it follows that $\dim(\operatorname{Ram}(f))\geq n-2$ as well.
\end{proof}

\section{A Topological Marikova Lemma} We now pause briefly from our work at hand and develop a technical result on definable groups that we will use later on. This is an adaptation of a method due to Marikova (\cite{MarikovaGps}). Let us first try to motivate the result.

Suppose that $(G,\cdot)$ a definable group in $\CR_V$ (so $G\subset R^n$, say). Then $G$ is both a group and a topological space, but is not necessarily a topological group. On the other hand, by definability in $\CR_V$, $G$ is `almost' a topological group -- i.e. the group operation and inverse are continuous outside a set of smaller dimension. Analogous statements hold replacing `continuous' with `differentiable', 'twice differentiable', and so on.

One could wonder, then, whether we can `fix' the topology on $G$, editing it slightly (in a definable way) to give a genuine topological group. An analogous situation was considered by Marikova (\cite{MarikovaGps}) in studying invariant groups in o-minimal structures. In the definable o-minimal context this was first proved by Pillay \cite{Pi5}, based on ideas of Hrushovski \cite{Bou} and v.d. Dries \cite{vdDriesGpChunk} and ultimately going back to Weil's group chunk theorem.  The essence of Marikova's argument is easier to distill and generalize than Pillay's, and it will best for our needs. Applications of Marikova's method (similar to those below) were also used in \cite{HaHaPeGps} to study interpretable groups in various valued fields, and in the M.Sc. thesis of Peterzil's student, M. Abd Elkareem. 

The main point of the method is that one can often retopologize the group so as to make it a local group (i.e. the operations are continuous in a neighborhood of the identity). We will need two variations of this statement in the ensuing sections. So, for convenience moving forward, we have decided to give it a general treatment. It may be worth pointing out that when the local group can be defined canonically enough, the topology obtained on the local group can often be extended to a group topology (on the whole group). See, e.g, \cite[Lemma 2.11]{HaHaPeGps} -- though as this is not necessary for our needs, we will not dwell on it.  

A natural general setting to present the method seems to be the following:

\begin{definition}\label{D: local group} Let $\CC$ be a subcategory of the category of topological spaces. Let $G$ be a group whose universe is also a topological space.
\begin{enumerate}
    \item $G$ is \textit{generically $\CC$} if there are dense open sets $U\subset G^2$ and $V\subset G$ such that:
    \begin{enumerate}
        \item $U$ and $V$ are objects in $\CC$.
        \item $U\cdot U\subset V$ and $V^{-1}=V$.
        \item The group operation from $U^2$ to $V$ is a morphism in $\CC$, and the inverse from $V$ to $V$ is a morphism in $\CC$. 
    \end{enumerate}
    \item $G$ is \textit{locally $\CC$} if there are open sets $1\in U,V\subset G$ such that:
    \begin{enumerate}
        \item $U$ and $V$ are objects in $\CC$.
        \item $U\cdot U\subset V$ and $V^{-1}=V$.
        \item The group operation from $U^2$ to $V$ is a morphism in $\CC$, and the inverse from $V$ to $V$ is a morphism in $\CC$.
    \end{enumerate}
\end{enumerate}
\end{definition}

Note that we are often interested in model-theoretic notions of genericity (whereas the definitions above use a purely topological such notion). On the other hand, for definable properties in model-theoretically tame topological settings, the phrase `outside a set of smaller rank' is typically synonymous with `on a dense open set' -- so the above formalism seems to capture various model-theoretic settings as well. 

We now present the result. Again, the argument is really due to Marikova.

\begin{notation} Suppose $G$ is a group (written multiplicatively), and $d\in G$. We let $G_d=(G,*_d)$ be the group $G$ seen through the permutation $x\mapsto dx$: that is, $x*_dy=xd^{-1}y$. So $d$ is the identity of $G_d$, and the inverse is $x^{-1_d}=dx^{-1}d$.
\end{notation}

\begin{theorem}\label{T: marikova} Let $\CC$ be a subcategory of the category of topological spaces. Assume that:
\begin{enumerate}
    \item All constant maps between objects of $\CC$ are morphisms.
    \item $\CC$ has finite products, given by cartesian products with the product topology.
    \item If $X$ is an object in $\CC$ and $Y\subset X$ is open, then $Y$ is an object in $\CC$ and the inclusion $Y\rightarrow X$ is a morphism.
\end{enumerate}
Then, if $G$ is generically $\CC$, there is a dense open $D\subset G$ such that $G_d$ is locally $\CC$ for each $d\in D$.
\end{theorem}

\begin{proof}
    The main point is to show that for a dense open set of $d\in G$, the function $(x,y,z)\mapsto xy^{-1}z$ is a morphism in $\CC$ in a neighborhood of $(d,d,d)$. To show this, we consider the following sequence of maps from $G^4\rightarrow G^4\rightarrow G^3\rightarrow G^2\rightarrow G$: $$(t,x,y,z)\mapsto(t,tx,y^{-1},z)\mapsto(t,txy^{-1},z)\mapsto(t,txy^{-1}z)\mapsto xy^{-1}z.$$ For $w,d\in G$, we further distinguish the sequence of points $$(w,d,d,d)\mapsto(w,wd,d^{-1},d)\mapsto(w,w,d)\mapsto(w,wd)\mapsto d$$ as above. We claim that for a dense open set of $(w,d)\in G^2$, all four of these maps are locally in $\CC$ at the relevant distinguished points as above. To see this, note that since $\CC$ has products given by the usual cartesian products, it suffices to check everything `one coordinate at a time'. Thus, we need only separately note that each of the following is locally in $\CC$ for dense open many $(w,d)$:
    \begin{enumerate}
        \item $(a,b)\mapsto ab$ at $(w,d)\mapsto wd$.
        \item $a\mapsto a^{-1}$ at $d\mapsto d^{-1}$.
        \item $(a,b)\mapsto ab$ at $(wd,d^{-1})\mapsto w$.
        \item $(a,b)\mapsto ab$ at $(w,d)\mapsto d$.
        \item $(a,b)\mapsto a^{-1}b$ at $(w,wd)\mapsto d$.
    \end{enumerate}
    Indeed, each of (1)-(5) holds on a dense open set as a straightforward consequence of the fact that $G$ is generically $\CC$.

    Now let $Z\subset G^2$ be a dense open set of $(w,d)$ on which the above maps are in $\CC$. Let $D$ be the projection of $Z$ onto the second coordinate. Then $D$ is dense open in $G$. We claim that $D$ satisfies the requirements of the theorem. Indeed, let $d\in D$ be arbitrary. Then there is $w$ with $(w,d)\in Z$, and thus the map $(t,x,y,z)\mapsto xy^{-1}z$ is in $\CC$ in a neighborhood of $(w,d,d,d)$. Using that $\CC$ contains all inclusions of open sets, we may assume this neighborhood of $(w,d,d,d)$ has the form $W\times U^3$ with $w\in W$ and $d\in U$. Then it follows that $(x,y,z)\mapsto xy^{-1}z$ is in $\CC$ on $U^3$, by factoring through $U^3\rightarrow W\times U^3$ (i.e. taking the product of the constant map $w:U^3\rightarrow W$ and the identity map $U^3\rightarrow U^3$; so here we use that $\CC$ contains all constant maps).

    Now given that $xyz\rightarrow xy^{-1}z$ is in $\CC$ on $U^3$, it follows that $G_d$ is locally $\CC$: indeed, recall that $x*_dy=xd^{-1}y$ and $x^{-1_d}=dx^{-1}d$. By a similar argument to above (using the constant map $d$), these maps are now in $\CC$ when we restrict $x$ and $y$ to $U$.
\end{proof}

\section{Locally Manifold-Loric Groups}

Let us now explain how Theorem \ref{T: marikova} will come into play. The proof in \cite{CasOmin} of the higher-dimensional case of the trichotomy splits into two cases, according to whether the given non-locally modular relic $\CM$ interprets a strongly minimal group. Let us sketch the case when there is such a group, say $G$. In this case, one first replaces $\CM$ with the structure $\CG$ it induces on $G$, thereby assuming $\CM$ is already an expansion of a group $(M,+)$ (this step requires elimination of imaginaries to turn $\CG$ into a definable relic). Since the background setting in \cite{CasOmin} is an o-minimal field, one can even assume $(M,+)$ is a topological group with the affine topology; then since the relevant lore in \cite{CasOmin} is the \textit{full} lore, it follows that $(M,+)$ is a \textit{loric group} (i.e. the underlying set, group operation and inverse maps are loric; for the full lore, this is the same as being a topological group). One then proves the desired bound $\dim(M)\leq 2$ by applying purity of ramification to a suitable restriction of the group operation $+:M^2\rightarrow M$.

In our case, things are more subtle. Suppose $\CM$ is a strongly minimal definable $\CR_V$-relic interpreting a strongly minimal group $G$. In order to run the proof sketched above, we would need that (1) $G$ is $\CR_V$-definable (i.e. if it is imaginary, it can be eliminated), and (2) the embedding into a power of $R$ can be chosen to make $G$ a loric group (in particular, a topological group). Unfortunately step (2) does not entirely work. In this section, we show how to modify it using Theorem \ref{T: marikova} without harming the argument in \cite{CasOmin}. 

First we point out that step (1) works as planned. Namely, the following is an immediate consequence of results in section 4 (see also \cite[Lemma 2.8]{HaOnPi}): 
\begin{fact}\label{F: interpretable is definable}
    Let $\CM$ be a strongly minimal definable $\CR_V$-relic. Then any strongly minimal set $D$ interpretable in $\CM$ is in $\CR_V$-interpretable bijection with an $\CR_V$-definable set.
\end{fact}
\begin{proof} By strong minimality there is an $\mathcal M$-interpretable finite correspondence between $M$ and $D$. Now apply Lemmas \ref{L: use of acl model} and \ref{L: e of fin i} (noting that $\CR_V$ has definable Skolem functions, so has the acl-model property).
\end{proof}

Now suppose we have a non-locally modular strongly minimal definable $\CR_V$-relic $(M,+,...)$ (expanding the group $(M,+)$). We would like $M$ to be a loric group. Unfortunately we do not know whether this can always be arranged. Instead, we work with the notion of a \textit{locally manifold-loric group}:



\begin{definition}\label{D: loc lor}
    Let $G$ be an $\CR_V$-definable group. We say that $G$ is \textit{locally manifold-loric} if there are open neighborhoods $U\subset G$ and $V\subset G$ of the identity such that:
    \begin{enumerate}
        \item $U^{-1}=U$, and $U\cdot U\subset V$.
        \item $U$ and $V$ are loric $\dim(G)$-manifolds.
        \item The group operation and inverse restrict to loric maps $U^2\rightarrow V$ and $U\rightarrow U$.
    \end{enumerate}
\end{definition}

Definition \ref{D: loc lor} coincides with $G$ having local dimension $\dim(G)$ at the identity, and also being `locally $\mathcal C$' as in Definition \ref{D: local group}, where $\mathcal C$ is the category of loric manifolds with loric maps. Note that `locally $\CC$' is not quite the right notion by itself, because this would allow any $G$ with an isolated point at the identity.

We now show:
\begin{lemma}\label{L: loc loric}
    Let $(G,\cdot)$ be an $\CR_V$-definable group. Then $G$ is $\CR_V$-definably isomorphic to a definable group $G'\subset R^n$ so that $G'$, with the subspace topology induced from $R^n$, is locally manifold-loric.
\end{lemma}
\begin{proof}
    Let $\CC$ be the category of loric manifolds with loric maps. Then $\CC$ satisfies the requirements of Theorem \ref{T: marikova}. Also, by the genericity property of lores, $G$ is a generically $\CC$-group. Thus, by Theorem \ref{T: marikova}, the group $G_d$ is locally $\CC$ for dense open-many $d\in G$. In particular, let $d\in G$ be generic. Then $G_d$ is locally loric. Moreover, by genericity, $G$ is a loric $\dim(G)$-manifold in a neighborhood of $d$, which implies that $G_d$ is locally manifold-loric.
    
    
    
\end{proof}

Theorem \ref{T: marikova} also implies the following, which is needed in the next section:

\begin{corollary}\label{C: finite torsion}
    In $\CR_V$ there are no infinite definable torsion groups.
\end{corollary}
\begin{proof} Let $(G,\cdot,e)$ be an infinite definable group. By compactness, it suffices to show that for each $k\in\mathbb Z^+$, the map $x\mapsto x^k$ is not the identity on all of $G$. 

This time we let $\CC$ be the category of all locally loric $C^1$-manifolds with loric $C^1$-maps. Then $\CC$ again satisfies the requirements of Theorem \ref{T: marikova}, and again $G$ is generically $\CC$. So by Theorem \ref{T: marikova} (arguing as above to ensure top-dimensionality), we may assume $G$ is locally $\CC$ and locally of dimension $\dim(G)$ at the identity. The result is that $(x,y)\mapsto xy$ and $x\mapsto x^{-1}$ restrict to $\CR$-definable $C^1$ maps of infinite $\CR$-definable $C^1$-manifolds in neighborhoods of $e$. But in this case, the power map $x\mapsto x^k$ induces scaling by $k$ on the tangent space $T_e(G)$; and since scaling by $k$ is surjective -- thus not constant by top-dimensionality -- $x\mapsto x^k$ cannot be constant either.
\end{proof}

We also get the following corollary, which is not needed for the sequel but may be of interest: 

\begin{corollary}
    If $\CR_V$ is power bounded then any definable topological group has a definable open subgroup definably isomorphic to a loric group. 
\end{corollary}
\begin{proof}
    Let $G$ be such a group. By Lemma \ref{L: loc loric}, we may assume $G$ is locally loric. By \cite[Proposition 6.3]{AcoLoGps}, $G$ has a neighborhood basis at the identity consisting of open subgroups. Since $G$ is locally loric, one of these subgroups is loric.
\end{proof}

It may be worth pointing out that specialised to the case of pure real closed valued fields, the above corollary extends \cite[Theorem A]{HruPil} to RCVF. We recall that this theorem asserts that a Nash group over $\Rr$ is locally Nash isomorphic to the $\Rr$-rational points $H(\Rr)$ of some algebraic group $H$. The formulation (and proof) of this theorem over arbitrary real closed fields is given in \cite{HrPi2011} (see the proof of Theorem 2.1 of that paper). We also note that, in the RCVF setting, the same conclusion holds of any definable group, if we drop the requirement that the definable isomorphism is Nash. 

\section{Higher Dimensional Relics of T-Convex Fields}
We are ready to prove the first part of our main theorem for $\CR_V$-relics:
\begin{theorem}\label{T: at most 2}
    Let $\CM$ be a strongly minimal non-locally modular definable $\CR_V$-relic. Then $\dim(M)\le 2$. 
\end{theorem}
\begin{proof}


    Adding constants, we may assume that $\acl_\CM(\0)$ is infinite. Passing to a reduct of $\CM$ (keeping non-local modularity), we may assume that the language of $\CM$ is countable, and that every $\CM(\emptyset)$-definable set is $\CR_V(\emptyset)$-definable. Thus, we may assume that our structures $\CR_V$ and $\CM$ meet the requirements of \cite[Assumption 3.1]{CasOmin} with respect to the lore $\CS$ of locally $\CR$-definable sets (The main point to note is that $\CR$ and $\CS$ satisfy manifold ramification purity of order $2$, which is given by Proposition \ref{P: purity success}). So we can now use freely the results of sections 3-7 of \cite{CasOmin}.
    
    If $\CM$ does not interpret a strongly minimal group, then we conclude by \cite[Theorem 7.1, Theorem 5.4]{CasOmin}. (These results collectively establish $\dim(M)\leq 2$ assuming (i) manifold ramification purity of order 2 with respect to the given lore, and (ii) that $\CM$ does not interpret a strongly minimal group.) 
    
    So, suppose $\CM$ interprets the strongly minimal group $G$. Then $\dim(G)=\dim(M)$ (as $G$ and $M$ are in finite correspondence). Moreover, by Fact \ref{F: interpretable is definable} and Lemma \ref{L: loc loric}, $G$ is $\CR_V$-definably isomorphic to a locally manifold-loric $\CR_V$-definable group. Thus, replacing $\CM$ by the structure it induces on $G$, we may assume $\mathcal M$ is already an expansion of a locally manifold-loric definable group. Since strongly minimal groups are abelian, we write this group as $(M,+)$. We also abuse notation (as in \cite{CasOmin}) and often use 0 as shorthand for $(0,0)\in M^2$. Since $(M,+)$ is locally manifold-loric, we fix open loric neighborhoods $U$ and $V$ of 0 as in Definition \ref{D: loc lor}. Without loss of generality, assume $U\subset V$.

    Our task is now to show that the analogous result in the o-minimal setting (\cite[Theorem 4.2]{CasOmin}) goes through for our group $(M,+)$. Let us recall the structure of the proof from \cite{CasOmin}. One can distill the argument into four pieces:
    \begin{enumerate}
        \item[I.] First, one constructs a particular $\CM$-definable strongly minimal plane curve $C\subset M^2$, chosen to satisfy the following (see \cite[Lemma 4.6]{CasOmin}):
    \begin{enumerate}
        \item The generic type of $C$ has finite stabilizer.
        \item $C\cap -C$ is finite.
        \item $0\in C$, and $C$ is a loric $\dim(M)$-manifold in a neighborhood of $0$.
    \end{enumerate}
    \item[II.] Second, the curve $C$ is used to construct a function $f:X\rightarrow Y$ with the following properties:
    \begin{enumerate}
        \item $X\subset C^2$ and $Y\subset M^2$ are open, loric $2\dim(M)$-manifolds and $f$ is the restriction of $+$ to $X$ (so $X+X\subset Y$).
        \item $f$ is loric and finite-to-one.
        \item $(0,0)\in X$, and $f$ ramifies at $(0,0)$.
    \end{enumerate}
    \item[III.] Next, one notes that by purity of ramification, the value $\dim(\operatorname{Ram}(f))$ (the dimension of the ramification locus of $f$) is at least $2\dim(M)-2$.
    \item[IV.] On the other hand, one shows using \textit{weak detection of closures} (in differentiable Hausdorff geometric structures) that at the same time, $\dim(\operatorname{Ram}(f))\leq\dim(M)$. Then combining with (III) gives $2\dim(M)-2\leq\dim(M)$, and thus $\dim(M)\leq 2$. 
    \end{enumerate}

    In the T-convex setting, steps III and IV go through word-for-word. That is, III is a direct consequence of purity of ramification (Proposition \ref{P: purity success}), and IV is a direct consequence of weak detection of closures (\cite[Proposition 3.40 and Theorem 4.10]{CaHaYe}).
    
    Steps I and II also go through, but with minor adaptations. Thus, we now discuss these steps in \cite{CasOmin} and how to change them to fit our needs.
   

    Let us begin with I. In \cite{CasOmin}, the construction of a curve $C$ satisfying the required properties proceeds as follows: first choose $C$ to satisfy (1) (which can always be done in a non-locally modular strongly minimal group, see \cite[Lemma 7.2 and Theorem 7.3]{cashasva}); then choose $a\in C$ generic and replace $C$ by $C-a$, and deduce (2) from the fact that there are no o-minimally definable infinite torsion groups. Finally, since (i) $C$ is a loric $\dim(M)$-manifold in a neighborhood of $a$, and (ii) translation by $-a$ is a loric homeomorphism, we get (3). 
    
    In our case, the proofs of (1) and (2) are unchanged, because one only uses that there are no infinite definable torsion groups -- a property which holds also in $\CR_V$ by Corollary \ref{C: finite torsion}. 

    To arrange (3), one just needs to choose our generic $a$ to belong to the distinguished open set $U$ that we have fixed from the beginning. This is possible because $\dim(U)=\dim(G)$. Then the definition of locally manifold-loric implies that translation by $-a$ is a loric homeomorphism from $U$ to $U-a$, and thus $C-a$ is a loric $\dim(M)$-manifold in the neighborhood $U-a$ of 0.

    Finally, as in \cite{CasOmin}, step II is done by choosing $X$ and $Y$ to be small enough neighborhoods of the zero vectors in $C^2$ and $M^2$. One can arrange II(a) and II(b) using that $M$ is locally manifold loric and $C\cap -C$ is finite (so we choose $X$ and $Y$ to be contained in $U^4$ and $V^2$, where $U$ and $V$ are the distinguished open sets witnessing that $M$ is locally manifold loric). Then II(c) follows as in \cite{CasOmin}, by noting that $0+0=0$, while for $0\neq u\in C$ near 0, one has the two preimages $0+u=u+0=u$ near $(0,0)$.

    

  So, since we can recover steps I and II, the rest of the proof works unchanged, and we get the same bound $\dim(M)\leq 2$ as in \cite{CasOmin}.    
\end{proof}


\section{Differentiable Hausdorff Geometric Fields}

For the rest of the paper, we show that one-dimensional definable strongly minimal relics of $\CR_V$ are locally modular, generalizing an earlier result of Hasson, Onshuus, and Peterzil in the o-minimal setting (\cite{HaOnPe}). As in that paper, our strategy is to assume non-local modularity and interpret a strongly minimal (thus algebraically closed) field -- a contradiction since there are no $\CR_V$-definable algebraically closed fields of dimension 1.

It will be convenient to give the argument a general treatment. In \cite{CaHaYe}, we showed that non-trivial one-dimensional strongly minimal definable relics of ACVF interpret a strongly minimal group. The main tool was a general theory of \textit{coherent slopes} that could be developed in the relic. We also remarked that a similar abstract strategy might produce a field (at least in characteristic zero), not just a group -- but more would be needed. The main purpose of the next two sections is to complete this program, giving an abstract construction of a field using coherent slopes. The axiomatic setting will essentially consist of a characteristic zero topological field with geometric theory and equipped with `tangent spaces' of definable sets where appropriate.

Unfortunately, there will be slight inconsistencies in notation and terminology with \cite{CaHaYe}. In that paper we considered a more general notion of $n$-slopes (for all $n\geq 1$), which could be collected into an inverse system of groupoids called the \textit{definable slopes} of a model of ACVF. A key requirement was that a curve is determined near a point by its $n$-slopes for all $n$ (i.e. the category of germs is the inverse limit of the categories of slopes). This property allowed us to construct a one-dimensional family of $n$-slopes at a point for some $n$ (and this family was used as a first approximation of the group we constructed).

The rather elaborate setup discussed above was designed to work in positive characteristic, where we could not control the value of $n$ witnessing a one-dimensional family as above. In contrast, in our characteristic zero setting, families of $n$-slopes always exist for $n=1$ (essentially by Sard's Theorem, see Lemma \ref{L: tangency slope is generic}). Thus, in this paper, we only use 1-slopes, and we do \textit{not} assume a curve is determined near a point by its slope. Thus, strictly speaking, we are \textit{leaving} the formalism in \cite{CaHaYe} and redeveloping it in a different context. For the most part, however, the two treatments are quite similar.

There will also be differences in notation, largely because in the present paper we work with \textit{slopes of types in all arities} (while in \cite{CaHaYe}, to handle the Frobenius effectively we were constrained to slopes of curves in $K^2$). Overall, the current approach seems more natural. Thus, we hope the reader will forgive us for needing to redevelop some notions from \cite{CaHaYe} in a different language -- we feel the outcome is worth the effort for aesthetic reasons.

We begin with our axiomatic framework. Recall that a \textit{differentiable Hausdorff geometric structure} is a Hausdorff geometric structure that comes equipped with a notion of smoothness and an associated notion of tangent spaces (see \cite[Definition 3.34]{CaHaYe}).

\begin{definition}\label{D: differentiable HGF}
    A \textit{differentiable Hausdorff geometric field} is a t-minimal differentiable Hausdorff geometric structure $(\CK,\tau)$ such that:
    \begin{enumerate}
        \item $\CK=(K,+,\cdot,...)$ is an expansion of a field, and $\tau$ is a field topology.
        \item The `scalar field' of the differential structure is $K$ (that is, tangent spaces are $K$-vector spaces).
        \item If $X\subset K^n$ is definable, and $a\in X$ is a smooth point, then $T_a(X)$ is an affine linear subspace of $K^n$, viewed as a $K$-vector space with zero vector $a$ (thus, we can and do view $T_a(X)$ as encoded by a tuple in some $K^m$ via Grassmanians -- see below for an elaboration on this).
        \item The smooth part and tangent space maps, $X\mapsto X^S$ and $(X,a)\mapsto T_a(X)$ (including the vector space structure), are $\CK(\emptyset)$-definable in families.
    \end{enumerate}
\end{definition}

\textbf{Fix $(\CK,\tau)$, a differentiable Hausdorff geometric field. We assume that $\CK$ is $\kappa$-saturated and $\kappa$-strongly homogenous for some sufficiently large $\kappa$. For now, all tuples are assumed to reside in powers of $K$.}

\begin{remark} Recall that we also still have our fixed T-convex structure $\CR_V$ that we will continue to mention as we go. The reader could assume $\CK=\CR_V$, but we use the different notation to distinguish which arguments are more general.
\end{remark}

We now continue by developing some basic facts and terminology in $\CK$. As noted above, we will view tangent spaces as tuples in powers of $K$. Let us elaborate: suppose $X\subset K^n$ is definable and smooth at $a\in X$. Then $T_a(X)$ is an affine linear subpace of $K^n$ containing $a$, viewed as a vector space with $a$ as its origin. We describe how to encode this definably. First, we view $a$ as a code for the $K$-vector space structure $V_a$ on $K^n$ with $a$ as its origin (so equipped with the addition operation $x+_ay=x+y-a$ and scaling maps $x\mapsto a+c(x-a)$ for $c\in K$). Then we view $T_a(X)$ as a $K$-vector subspace of $V_a$ -- an object that can be encoded in a straightforward (and uniform) way via Grassmanians in $V_a$. So, strictly speaking, the code of $T_a(X)$ is the pair $(a,s)$ where $s$ codes the subspace $T_a(X)$ of $V_a$. Note, in particular, that with this formalism we will always have $a\in\dcl(T_a(X))$ (since $a$ is included in the encoding).

As in \cite{CaHaYe}, the frontier inequality implies that the \textit{germ} of a complete type at a realization is well-defined, and thus so is the tangent space of a complete type at a realization. We use the notation $G(a/A)$ and $T(a/A)$ to denote the germ and tangent space of $\tp(a/A)$ at $a$.

In \cite{CaHaYe}, we developed a theory of \textit{weakly generic intersections} between two families of plane curves in a Hausdorff geometric structure. This allowed us to formalize various geometric properties we needed. In what follows, we develop (simpler) analogs at the level of types.

\begin{definition}\label{D: family of curves}
    An \textit{abstract family of plane curves} is a complete type $p=\tp(a,t/A)$ such that $\dim(a/A)=2$ and $\dim(a/At)=1$.
\end{definition}

\begin{remark}
Definition \ref{D: family of curves} makes sense in any geometric structure. In the context of strongly minimal structures, it is essentially equivalent to the usual families of plane curves one considers. That is, suppose $\CM$ is strongly minimal. Let $T$ be an infinite definable set, and suppose that $\{C_t:t\in T\}$ is an $A$-definable \textit{almost faithful} family of plane curves (i.e. one-dimensional subsets of $M^2$; see \cite[\S2]{CasACF0} for a discussion of almost faithfulness). Let $t\in T$ be $A$-generic, and let $a\in C_t$ be $At$-generic. Then $p=\tp(a,t/A)$ is an abstract family of plane curves. Moreover, if $q=\tp(b,u,B)$ is any abstract family of plane curves in $\CM$, and $\tp(b/Bu)$ is stationary, then $q$ essentially (i.e. up to reparametrization) arises from a construction of the above form.
\end{remark}

\begin{definition}
    Let $p=\tp(a,t/A)$ and $q=\tp(b,u/A)$ be abstract families of plane curves.
    \begin{enumerate}
        \item A \textit{$(p,q)$-intersection} is a triple $(a,t,u)$ such that $(a,t)\models p$ and $(a,u)\models q$.
        \item A \textit{strong $(p,q)$-intersection} is a $(p,q)$-intersection $(a,t,u)$ such that $t$ and $u$ are independent over both $A$ and $Aa$.
        \item A $(p,q)$-intersection $(a,t,u)$ is \textit{strongly approximable} if every neighborhood of $(a,t,u)$ contains a strong $(p,q)$-intersection.
        \item A \textit{$(p,q)$-tangency} is a strongly approximable $(p,q)$-intersection $(a,t,u)$ such that $T(a/At)=T(a/Au)$.
        \item A \textit{$(p,q)$-multiple intersection} is a $(p,q)$-intersection $(a,t,u)$ such that every neighborhood of $(a,t,u)$ contains $(p,q)$-intersections $(a',t',u'),(a'',t',u')$ with $a'\neq a''$.
    \end{enumerate}
\end{definition}

Aside from Definition \ref{D: differentiable HGF}, we need one more assumption before building a field -- again adapted to types from an original version in \cite{CaHaYe}:

\begin{definition}\label{D: TIMI} Let $p=\tp(a,t/A)$ and $q=\tp(b,u/A)$ be abstract families of plane curves. We say that the pair $(p,q)$ \textit{satisfies TIMI} if every $(p,q)$-tangency is a $(p,q)$-multiple intersection. If this holds for all such $p$ and $q$, we say that $(\CK,\tau)$ \textit{satisfies TIMI}.
\end{definition}

TIMI is short for `tangent intersections are multiple intersections'. It is the key geometric data allowing us to recover differential data from topological data in a relic.

\subsection{Finding Infinitely Many Slopes}

As described above, to interpret a field in a $\CK$-relic, we will need a family of plane curves achieving infinitely many distinct slopes at a generic point. In the current setting, the existence of such families is essentially a consequence of \textit{Sard's Theorem}. More precisely, recall that a version of Sard's Theorem is built into the axioms of a differentiable Hausdorff structure: in the current language, the statement we need is the following:

\begin{fact}\cite[Lemma 3.38(3)]{CaHaYe}\label{F: Sard} Fix any $a\in R^m$, $b\in R^n$, and $A$. Let $0_b$ be the zero vector in $T(b/A)$. Then the projection $T(ab/A)\rightarrow T(b/A)$ is surjective with kernel $T(a/Ab)\times\{0_b\}$.
\end{fact}

Fact \ref{F: Sard} will be used crucially in Lemma \ref{L: strong intersections are transverse}. From now on, we will refer to this fact simply as \textit{Sard's Theorem}.

In fact, the ingredients used to construct a family of slopes are also crucial for proving TIMI in our T-convex field $\CR_V$. Thus, we will now give these ingredients a more general treatment, before deducing both of the desired consequences.

First we show that strong intersections of plane curves are `transverse' (have different tangent spaces):

\begin{lemma}\label{L: strong intersections are transverse} Let $p=\tp(a,t/A)$ and $q=\tp(b,u/A)$ be abstract families of plane curves. Let $(a,t,u)$ be a strong $(p,q)$-intersection. Then $T(a/At)\neq T(a/Au)$.
\end{lemma}

\begin{proof}
    We proceed with two key claims. The first is the main use of strongness.
    
    \begin{claim} The projection $T(atu/A)\rightarrow T(tu/A)$ is an isomorphism.
    \end{claim}
    \begin{proof} By Sard's Theorem, $T(atu/A)\rightarrow T(tu/A)$ is surjective. So it will suffice to show that $\dim(atu/A)\leq\dim(tu/A)$. Since $\dim(a/A)=2$ and $\dim(a/At)=1$, it follows that $\dim(t/Aa)=\dim(t/A)-1$. Similarly, $\dim(u/Aa)=\dim(u/A)-1$. Thus, $$\dim(atu/A)=\dim(a/A)+\dim(tu/Aa)\leq\dim(a/A)+\dim(t/Aa)+\dim(u/Aa)$$ $$=2+(\dim(t/A)-1)+(\dim(u/A)-1)=\dim(t/A)+\dim(u/A).$$ But by strongness, $\dim(t/A)+\dim(u/A)=\dim(tu/A)$, which implies the claim. 
\end{proof}

\begin{claim} $T(atu/A)$ is the fiber product of $T(at/A)$ and $T(au/A)$ over $T(a/A)$.
\end{claim}
\begin{proof} There is a natural injection of $T(atu/A)$ into the given fiber product. So, as in the previous claim, it suffices to check that the dimensions are equal. That is, we want to show that $$\dim(atu/A)=\dim(at/A)+\dim(au/A)-\dim(a/A).$$ By the argument in the previous claim, both sides equal $\dim(t/A)+\dim(u/A)$.
\end{proof}
Now suppose $T(a/At)=T(a/Au)$. So there is $0\neq v\in T(a/At)\cap T(a/Au)$. Let $0_t$ and $0_u$ be the zero vectors in $T(t/A)$ and $T(u/A)$, respectively. Then by Sard's Theorem, $(v,0_t)\in T(at/A)$ and $(v,0_u)\in T(au/A)$. So by the second claim, $(v,0_t,0_u)\in T(atu/A)$. But since $v\neq 0$, this contradicts the first claim. 
\end{proof}

The next lemma says that strongly approximable intersections can be strongly approximated without moving the `plane' coordinate:

\begin{lemma}\label{L: approximation without moving point}
    Let $p=\tp(\hat a,\hat t/A)$ and $q=\tp(\hat b,\hat u/A)$ be abstract families of plane curves, so that $(\hat a,\hat t,\hat u)$ be a strongly approximable $(p,q)$-intersection. Let $U_{\hat t}\times U_{\hat u}$ be a neighborhood of $(\hat t,\hat u)$ (in the ambient power of $K$). Then there are $(t,u)\in U_{\hat t}\times U_{\hat u}$ so that $(\hat a,t,u)$ is a strong $(p,q)$-intersection.
\end{lemma}
\begin{proof}
      Suppose $U_{\hat t}$ and $U_{\hat u}$ are definable over $B$. Shrinking if necessary, we may assume $\hat a\hat t\hat u$ is independent from $B$ over $A$.
      
      Suppose the lemma fails. Note that the set of strong $(p,q)$-intersections is type-definable over $A$. Using this and compactness, one finds a formula $\phi(a,t,u)$ over $A$ such that:
        \begin{enumerate}
            \item Every strong $(p,q)$-intersection satisfies $\phi$.
            \item There are no $(t,u)\in U_{\hat t}\times U_{\hat u}$ with $(\hat a,t,u)\models\phi$.
        \end{enumerate}
        Note that (2) is a $B$-definable property of $\hat a$. Since $\hat a$ is independent from $B$ over $A$, $\tp(\hat a/A)$ and $\tp(\hat a/B)$ have the same germ at $\hat a$ (see \cite[Lemma 3.21]{CaHaYe}). Thus, there is a neighborhood $U_{\hat a}$ of $\hat a$ such that (2) holds for all $a\in U_{\hat a}$ realizing $\tp(\hat a/A)$. But then the neighborhood $U_{\hat a}\times U_{\hat t}\times U_{\hat u}$ has no strong $(p,q)$-intersections, contradicting that $(\hat a,\hat t,\hat u)$ is strongly approximable.
    \end{proof}

    We now arrive at our main goal. As described above, Lemma \ref{L: tangency slope is generic} allows us to construct infinite families of slopes at a point.

\begin{lemma}\label{L: tangency slope is generic} Let $p=\tp(a,t/A)$ and $q=\tp(b,u/A)$ be abstract families of plane curves. Let $(\hat a,\hat t,\hat u)$ be a $(p,q)$-tangency, and let $T=T(\hat a/A\hat t)=T(\hat a/A\hat u)$. Then $T\notin\acl(A\hat a)$.
\end{lemma}
\begin{proof} 
       Assume that $T\in\acl(A\hat a)$. Then $\hat t$ is independent from $T$ over $Aa$, so the types $\tp(\hat t/A\hat aT)$ and $\tp(\hat t/A\hat a)$ have the same germ at $\hat t$. It follows that in some neighborhood $U_{\hat t}$ of $\hat t$, every $t\models\tp(\hat t/A\hat a)$ satisfies $T(\hat a/At)=T$. Similarly, in some neighborhood $U_{\hat u}$ of $\hat u$, every $u\models\tp(\hat u/A\hat a)$ satisfies $T(\hat a/A\hat u)=T$. By Lemma \ref{L: approximation without moving point}, there is a pair $(t,u)\in U_{\hat t}\times U_{\hat u}$ so that $(\hat a,t,u)$ is a strong $(p,q)$-intersection. But then $T(\hat a/At)=T(\hat a/Au)$, contradicting Lemma \ref{L: strong intersections are transverse}. 

\end{proof}

\subsection{TIMI in T-convex Fields}

As promised, we now use Lemma \ref{L: tangency slope is generic} to prove that our T-convex field $\CR_V$ satisfies TIMI. 

\begin{remark}\label{R: tangent space clarification} Below, we will use that the tangent spaces we take in $\CR_V$ are the usual (geometrically meaningful) ones. That is, recall that we say a definable set $X$ is \textit{smooth} at $a\in X$ if $X$ restricts to an $\CR$-definable $C^1$-manifold of dimension $\dim(X)$ in a neighborhood of $a$; and then $T_a(X)$ is the standard tangent space of that manifold from o-minimal geometry. In particular, the tangent space of the graph of a $C^1$ function $f:R\rightarrow R$ at a point $a\in R$ is the line through $(a,f(a))$ with slope $f'(a)$.
\end{remark}

\begin{theorem}\label{T: t-convex timi} $\CR_V$ satisfies TIMI.
\end{theorem}
\begin{proof} Let $p=\tp(\hat a,\hat t/A)$ and $q=\tp(\hat a,\hat u/A)$ be abstract families of plane curves, where $(\hat a,\hat t,\hat u)$ is a $(p,q)$-tangency. We show that $(\hat a,\hat t,\hat u)$ is a $(p,q)$-multiple intersection. Absorbing parameters, we assume $A=\emptyset$. The first stage of the proof will be a lengthy series of reductions.

We first note that we can reduce the problem to curves in $R^2$ (in other words, we can assume $\hat a\in R^2$). The idea is that since $\dim(\hat a)=2$, some neighborhood of $\hat a$ can be identified with an open set in $R^2$. Since the statement of the theorem is local, we only need the germ $G(\hat a)$, so we replace the ambient space with $R^2$. If the reader wants, we sketch how to do this formally: first, since $\dim(\hat a)=2$, there is $b\in R^2$ interdefinable with $\hat a$. Then there is a $\emptyset$-definable map sending $\hat a$ to $b$ and witnessing interdefinability, and this map induces both a homeomorphism between neighborhoods of $\hat a$ and $b$, and also an isomorphism of tangent spaces $T(\hat a/\emptyset)\rightarrow T(b/\emptyset)$. It follows that $(b,\hat t,\hat u)$ is a $(p',q')$-tangency, where $p'=\tp(b,\hat t)$ and $q'=\tp(b,\hat u)$; and moreover, if $(b,\hat t,\hat u)$ is a $(p',q')$-multiple intersection, then $(\hat a,\hat t,\hat u)$ is a $(p,q)$-multiple intersection. Thus, we may harmlessly replace $\hat a$ with $b$ and assume $\hat a\in R^2$.

Let $\hat a=(\hat a_1,\hat a_2)$. So either $\dcl(\hat a_1\hat t)=\dcl(\hat a_2\hat t)$, or the germ $G(\hat a/\hat t)$ is a (germ of a) horizontal or vertical line (and similarly for $\tp(\hat a/\hat u)$). Applying a $\emptyset$-definable rotation if necessary (arguing as in the previous paragraph), we may assume $\dcl(\hat a_1\hat t)=\dcl(\hat a_2\hat t)$ and $\dcl(\hat a_1\hat u)=\dcl(\hat a_2\hat u)$. It follows that $p$ has a d-approximation at $(\hat a,\hat t)$ given by a $\emptyset$-definable family of bijections between neighborhoods of $\hat a_1$ and $\hat a_2$. Call this family $f$ -- so we write that $f_{\hat t}(\hat a_1)=\hat a_2$. Similarly, let $g$ be an analogous family giving a d-approximation of $\tp(\hat a,\hat u)$ -- so $g_{\hat u}(\hat a_1)=\hat a_2$. 

Next note that $(\hat a,\hat t)$ and $(\hat a,\hat u)$ are generic points in the domains of $f$ and $g$, respectively. So by the genericity property of lores, $f$ and $g$ are loric (i.e. locally $\CR$-definable) in neighborhoods of $(\hat a,\hat t)$ and $(\hat a,\hat u)$. Since we are only interested in the germs of $f$ and $g$ at these points, it follows that we may assume each of $f$ and $g$ is $\CR$-definable (that is, we have now reduced from the T-convex setting to the o-minimal setting).

We are now ready to give the argument. The main point is:

\begin{lemma}\label{L: main timi lemma} Every neighborhood of $(\hat a,\hat t,\hat u)$ contains a point $(x,y,t,u)$ with $(x,y)\neq(\hat a_1,\hat a_2)$, $f_t(\hat a_1)=g_u(\hat a_1)=\hat a_2$, and $f_t(x)=g_u(x)=y$.
\end{lemma}
\begin{proof} First, if the equation $f_{\hat t}(x)=g_{\hat u}(x)=y$ has infinitely many solutions in every neighborhood of $\hat a$, the lemma follows trivially (setting $(t,u)=(\hat t,\hat u)$). So assume $\hat a$ is an isolated solution of $f_{\hat t}=g_{\hat u}$. By o-minimality, without loss of generality there is $z>\hat a_1$ so that for $x\in(\hat a_1,z)$ we have $f_{\hat t}(x)<g_{\hat u}(x)$ (the other case is symmetric).

Let $Y$ be a d-approximation of $\tp(\hat t/\hat a)$. Shrinking if necessary, we may assume that $f_t(\hat a_1)=\hat a_2$ for all $t\in Y$. Since $t$ is generic in $Y$ over $\hat a$, we may also assume after shrinking that $Y$ is $\CR$-definable and $\CR$-definably connected in a neighborhood of $\hat t$. We now proceed with the following two claims:
\begin{claim} In any neighborhood of $(\hat a_1,\hat t)$, we can find $(x,t)$ with $t\in Y$, $x>\hat a_1$, and $f_t(x)<g_{\hat u}(x)$.
\end{claim}
\begin{proof} Choose $(x,t)=(x,\hat t)$ where $x>\hat a_1$ is sufficiently close to $\hat a_1$.
\end{proof}
\begin{claim} In any neighborhood of $(\hat a_1,\hat t)$, we can find $(x,t)$ with $t\in Y$, $x>\hat a_1$, and $f_t(x)>g_{\hat u}(x)$.
\end{claim}
\begin{proof} The idea is to rotate the graph of $f_{\hat t}$ counterclockwise about $\hat a$ until it `crosses' $g_{\hat u}$ to the right of $\hat a_1$. More precisely, we want to show that we can find $t$ close to $\hat t$ so that $f_t$ goes through the point $\hat a$ but with higher slope than $f_{\hat t}$ and $g_{\hat u}$. That this is possible follows from Lemma \ref{L: tangency slope is generic}.

Formally, Let $h$ be the $\hat a$-definable partial function $t\mapsto(f_t)'(\hat a_1)$, defined on a neighborhood of $\hat t$ in $Y$. In this language, the content of Lemma \ref{L: tangency slope is generic} is that $h(\hat t)$ is generic in $R$ over $\hat a$. By o-minimality, it follows that $h$ is open in a neighborhood of $\hat t$. Thus, we can find $t\in Y$ arbitrarily close to $\hat t$ with $h(t)>h(\hat t)$. Then for such $t$, we have (i) $f_t(\hat a_1)=f_{\hat t}(\hat a_1)=g_{\hat u}(\hat a_1)$ and (ii) $(f_t)'(\hat a_1)>(f_{\hat t})'(\hat a_1)=(g_{\hat u})'(\hat a_1)$ (where the last equality is because $(\hat a,\hat t,u)=(\hat a,\hat t,\hat u)$ is a $(p,q)$-tangency). By (i) and (ii), we can find $x>\hat a_1$ arbitrarily close to $\hat a_1$ with $f_t(x)>g_{\hat u}(x)$, which (as $t\rightarrow\hat t$ inside $Y$) is enough to prove the claim.
\end{proof}

We now finish the proof of Lemma \ref{L: main timi lemma}. Let $U_1$, $U_2$, $U_3$, and $U_4$ be any neighborhoods of $\hat a_1$, $\hat a_2$, $\hat t$, and $\hat u$, respectively. Shrinking if necessary, may assume each of the following:
\begin{enumerate}
    \item Each $U_i$ is $\CR$-definable and $\CR$-definably connected, and $U_3\cap Y$ is also $\CR$-definably connected (this is possible because $\hat t$ is generic in $Y$).
    \item $f$ and $g$ are defined and continuous on $U_1\times U_3$ and $U_1\times U_4$, respectively.
    \item $f(U_1\times U_3),f(U_1\times U_4)\subset U_2$.
\end{enumerate}
Let $c>\hat a_1$ be such that the interval $(\hat a_1,c)$ is contained in $U_1$, and let $Z=(\hat a_1,c)\times(U_3\cap Y)$. Then $Z$ is $\CR$-definable and $\CR$-definably connected, and the map $j:Z\rightarrow R$, $j(x,t)=f_t(x)-g_{\hat u}(x)$, is $\CR$-definable and continuous on $Z$. By the two claims above, $j$ takes both positive and negative values on $Z$ -- so by definable connectedness, there is $(x,t)\in Z$ with $j(x,t)=0$. Then $f_t(x)=g_{\hat u}(x)=y\in U_2$, say (by (3) above). Moreover, $f_t(\hat a_1)=g_{\hat u}(\hat a_1)=\hat a_2$ (since $t\in Y$) and $(x,y)\neq(\hat a_1,\hat a_2)$ (since $x>\hat a_1$). Thus $(x,y,t,\hat u)$ is the desired tuple to prove Lemma 10.12.
\end{proof}

We now return to the proof of Theorem \ref{T: t-convex timi}. By Lemma \ref{L: main timi lemma}, every neighborhood of $(\hat a_1,a_2,\hat t,\hat u)$ contains a point $(x,y,t,u)$ with $(x,y)\neq(\hat a_1,\hat a_2)$, $f_t(\hat a_1)=g_u(\hat a_1)=\hat a_2$, and $f_t(x)=g_u(x)=y$. But recall that $f$ and $g$ denoted arbitrarily small d-approximations of $p$ and $q$ at $(\hat a,\hat t)$ and $(\hat a,\hat u)$, respectively. By compactness, it now follows that any neighborhood of $(\hat a_1,\hat a_2,\hat t,\hat u)$ contains a point $(x,y,t,u)$ with $(x,y)\neq(\hat a_1,\hat a_2)$, $(\hat a_1,\hat a_2,t)\models p$, $(\hat a_1,\hat a_2,u)\models q$, $(x,y,t)\models p$, and $(x,y,u)\models q$. This shows that $(\hat a,\hat t, \hat u)$ is a $(p,q)$-multiple intersection.
\end{proof}

\subsection{Slopes}

We now return to our structure $\CK$. As promised in the introduction to this section, we proceed to develop the notion of a slope of a map of types at a realization. This formalism in particular leads to an elegant version of the chain rule (Lemma \ref{L: chain}), which is the main geometric fact we will use to encode the field operations.

\begin{definition}\label{D: slope}
    Let $a$ and $b$ be tuples, and suppose $b\in\acl(Aa)$. By Sard's Theorem, $T(ab/A)$ is the graph of a linear map $T(a/A)\rightarrow T(b/A)$. We call this map the \textit{slope from $a$ to $b$ over $A$}, denoted $\partial\left(\frac{a\rightarrow b}{A}\right)$.
\end{definition}

As with tangent spaces, we can view $\alpha=\partial\left(\frac{a\rightarrow b}{A}\right)$ as a tuple in $\CK$; note that $(a,b)\in\dcl(\alpha)$, and $\alpha\in\dcl(Aab)$. Note that tangent spaces are invariant under adding independent parameters (i.e. $T(a/A)=T(a/B)$ if $B\supset A$ is independent from $a$ over $A$; this follows since $\tp(a/A)$ and $\tp(a/B)$ have the same germ at $a$). Thus, we have:

\begin{lemma}\label{L: independent slope} If $b\in\acl(Aa)$ and $B\supset A$ is independent from $ab$ over $A$, then $\partial\left(\frac{a\rightarrow b}{A}\right)=\partial\left(\frac{a\rightarrow b}{B}\right)$.
\end{lemma}

We now give the main result of this subsection:

\begin{lemma}[Chain Rule]\label{L: chain} Let $a,b_1,...,b_n,c$ be tuples, and $A$ a parameter set. Assume that each $b_i\in\acl(Aa)$, and that $c\in\acl(Ab_1...b_n)$. Then $$\partial\left(\frac{a\rightarrow c}{A}\right)=\sum_{i=1}^n\partial\left(\frac{b_i\rightarrow c}{A\cup\{b_j:j\neq i\}}\right)\circ \partial\left(\frac{a\rightarrow b_i}{A}\right).$$
\end{lemma}
\begin{proof} Let $u\in T(a/A)$ and $w$ its image in $T(c/A)$ under $\partial\left(\frac{a\rightarrow c}{A}\right)$ (so $(u,w)\in T(ac/A)$). For each $i$, let $v_i$ be the image of $u$ under $\partial\left(\frac{a\rightarrow b_i}{A}\right)$, and let $w_i$ be the image of $v_i$ under $\partial\left(\frac{b_i\rightarrow c}{A\cup\{b_j:j\neq i\}}\right)$. Finally, let $w'=w_1+...+w_n$. So we want to show that $w'=w$.

By Sard's Theorem applied to $T(a,b_1,...,b_n,c/A)\rightarrow T(a,c/A)$, there are $(v_1',...,v_n')$ with $(u,v_1',...,v_n',w)\in T(a,b_1,...,b_n,c/A)$. For each $i$, projecting to $T(a,b_i/A)$ yields that $(u,v_i')\in T(a,b_i/A)$, and thus $v_i'=v_i$.

Next, for each $i$, there is an inclusion $T(b_i,c/A\cup\{b_j:j\neq i\})\hookrightarrow T(b_i,c/A)$, which shows that $(v_i,w_i)\in T(b_i,c/A)$. Let $z_i\in T(b_1/A)\times...\times T(b_n/A)$ be the vector with $v_i$ in the $i$th coordinate and zeros elsewhere. Since $T(b_1,...,b_n,c/A)\rightarrow T(b_i,c/A)$ is just the projection, we get that $(z_i,w_i)\in T(b_1,...,b_n,c/A)$ for each $i$. Adding these together then yields $(v_1,...,v_n,w')\in T(b_1,...,b_n,c/A)$.

On the other hand, since $T(a,b_1,...,b_n,c/A)\rightarrow T(b_1,...,b_n,c/A)$ is the projection, we have $(v_1,...,v_n,w)\in T(b_1,...,b_n,c/A)$. But since $c\in\acl(Ab_1...b_n)$, the projection $T(b_1,...,b_n,c/A)\rightarrow T(b_1,...,b_n/A)$ is in isomorphism, which shows that $w'=w$.
\end{proof}

\section{Interpreting a Field}

We now develop the theory of coherent slopes in a $\CK$-relic, analogously to \cite{CaHaYe}. In \cite{CaHaYe}, we only defined coherent slopes in $K^2$, then used local homeomorphisms $M\rightarrow K$ to transfer slopes back to $M^2$. In the present context, we have access to a more flexible notion of slope, and we do not need to transfer with local homeomorphisms. Thus, our main task is to redevelop the basics of the analogous section of \cite{CaHaYe} in what we see as a more natural setting. 

\textbf{Throughout this section, fix $\CM$, a non-locally modular strongly minimal definable $\CK$-relic, with $\dim(M)=1$. All tuples are now taken in $\CM^{eq}$. Adding parameters, we assume every $\CM(\emptyset)$-definable set is $\CK(\0)$-definable, and that $\acl_{\CM}(\0)$ is infinite. We use $\dim$ for dimension in $\CK$, and $\rk$ for dimension in $\CM$.}

\subsection{Slopes in $\CM$} Following \cite{CasACF0} and \cite{CaHaYe}, we define:

\begin{definition}
    A tuple $a$ is \textit{coherent over $A$} if $\dim(a/A)=\rk(a/A)$. An infinite set is coherent over $A$ if each finite subset of it is. A set is coherent if it is coherent over $\emptyset$.
\end{definition}

We will freely use basic properties of coherence (see e.g. \cite[Lemma 4.13]{CaHaYe}). Most importantly, we note that coherence is preserved under $\mathcal M$-algebraicity: if $a$ is coherent over $A$, and $b\in\acl_{\mathcal M}(Aa)$, then $b$ is coherent over $A$.

\begin{definition} Let $(x,y)\in M^2$ be generic over $\emptyset$. A \textit{coherent slope at $(x,y)$} is a slope of the form $\alpha=\partial\left(\frac{x\rightarrow y}{A}\right)$ where:
    \begin{enumerate}
        \item $\rk(x/A)=\rk(y/A)=\rk(xy/A)=1$. 
        \item $Axy$ is coherent.
    \end{enumerate}
In this case:
\begin{itemize}
    \item $\alpha$ is an \textit{algebraic coherent slope} if $\alpha\in\acl(xy)$. Otherwise $\alpha$ is a \textit{non-algebraic coherent slope}.
    \item If $c=\operatorname{Cb}(\operatorname{stp(xy/A)})$ (i.e. $c$ is a canonical base of the strong type of $xy$ over $A$ in the sense of $\CM$), then $c$ is called a \textit{coherent representative of $\alpha$}.
    \end{itemize}
\end{definition}

\subsection{Detection of Tangency and Codes}

In \cite{CaHaYe}, we defined what it means for $\CM$ to \textit{detect tangency}. The main associated facts were (1) detection of tangency follows if $\CK$ satisfies TIMI, and (2) detection of tangency implies that every coherent slope has a \textit{code} (an associated tuple in $\CM$ allowing us to talk about that slope in $\CM$). Let us transfer these notions to the current context.

\begin{definition}
    $\CM$ \textit{detects tangency} if whenever $(x,y)\in M^2$ is generic, and $\alpha$ is a non-algebraic coherent slope at $(x,y)$, then any two coherent representatives of $\alpha$ are $\CM$-dependent over $xy$. 
\end{definition}

\begin{proposition}\label{P: detection of tangency}
    If $(\CK,\tau)$ satisfies TIMI, then $\CM$ detects tangency.
\end{proposition}
\begin{proof} Exactly as \cite[Theorem 7.52]{CaHaYe}, adapted to the current terminology. To summarize, let $c_1'$ and $c_2'$ be coherent representatives of a non-algebraic coherent slope at $(x,y)$. Choose real tuples $c_1,c_2$ from $\CM$ that are $\CM$-interalgebraic with $c_1'$ and $c_2'$, respectively. Then $p=\tp(xy,c_1)$ and $q=\tp(xy,c_2)$ are abstract families of plane curves, and $(xy,c_1,c_2)$ is a $(p,q)$-tangency. It follows by TIMI that $(xy,c_1,c_2)$ is a $(p,q)$-multiple intersection. One then applies \textit{detection of multiple intersections} (\cite[Proposition 3.40 and Theorem 6.10]{CaHaYe}) to conclude that $c_1$ and $c_2$ are $\CM$-dependent over $xy$, and thus so are $c_1',c_2'$.
\end{proof}

\begin{definition}
    Let $\alpha$ be a coherent slope at $(x,y)$. A \textit{code} for $\alpha$ is a tuple $s\in\CM^{eq}$ such that:
    \begin{enumerate}
        \item $s$ is $\CK$-interalgebraic with $\alpha xy$.
        \item $(x,y)\in\acl_{\CM}(s)$.
        \item For every coherent representative $c$ of $\alpha$, $s\in\acl_{\CM}(cxy)$.
    \end{enumerate}
\end{definition}

The idea is that the code of $\alpha$ (if it exists) allows us to talk about $\alpha$ in the language of $\CM$.

\begin{proposition}\label{P: codes exist}
    If $\CM$ detects tangency, then every coherent slope has a code.
\end{proposition}

\begin{proof}
    Exactly as in \cite[Proposition 7.45]{CaHaYe}. Let $\alpha$ be a coherent slope at $(x,y)$. If $\alpha$ is algebraic, then $(x,y)$ is a code. If not, let $c_1,c_2$ be independent coherent representatives. By detection of tangency, $c_1$ and $c_2$ are $\CM$-dependent over $xy$. Let $b$ be a canonical base of the relation witnessing this, i.e. $b=\operatorname{Cb}(\stp_{\CM}(c_2/c_1xy))$. Then by a forking computation, $bxy$ is a code of $\alpha$.
\end{proof}

\subsection{Recovering the Operations}

As described above, codes allow us to talk about coherent slopes in the language of $\CM$. Using this idea, the next step is to (approximately) recover the field operations on coherent slopes that have codes. The following is analogous to \cite[Proposition 7.29]{CaHaYe}:

\begin{lemma}\label{L: slope composition} Let $(x,y,z)\in M^3$ be generic, and let $\alpha_1$ and $\alpha_2$ be coherent slopes at $(x,y)$ and $(y,z)$, respectively. Assume that $\alpha_1$ is independent from $z$ over $xy$; $\alpha_2$ is independent from $x$ over $yz$; and $\alpha_1,\alpha_2$ are independent from each other over $xyz$. Then $\alpha_3=\alpha_2\circ\alpha_1$ is a coherent slope at $(x,z)$, and if $s_i$ is a code of $\alpha_i$, then $s_1s_2s_3$ is coherent and $s_3\in\acl_{\mathcal M}(s_1s_2)$.
\end{lemma}

\begin{proof} Choose coherent representatives $c_i$ of $\alpha_i$ for $i=1,2$. Using the various independence assumptions above to choose the $c_i$ as generically as possible, we may ensure that $c_1c_2xyz$ is coherent and $\dim(xyz/c_1c_2)=1$. By Lemma \ref{L: independent slope}, $\partial\left(\frac{x\rightarrow y}{c_1c_2}\right)=\alpha_1$ and $\partial\left(\frac{y\rightarrow z}{c_1c_2}\right)=\alpha_2$. So by the chain rule (Lemma \ref{L: chain}), $\partial\left(\frac{x\rightarrow z}{c_1c_2}\right)=\alpha_3$. But $c_1c_2xz$ is coherent, so by definition $\alpha_3$ is coherent. 

Now assume $s_i$ is a code for each $\alpha_i$. Then $s_1s_2s_3\in\acl_{\CM}(c_1c_2xyz)$, so $s_1s_2s_3$ is coherent. But $s_3\in\acl(s_1s_2)$ (applying the chain rule to  $\alpha_3=\alpha_2\circ\alpha_1$), so by coherence, $s_3\in\acl_{\CM}(s_1s_2)$.
\end{proof}

Note that a simpler version of the above argument proves the same result for inverses of slopes:

\begin{lemma}\label{L: inverse slope} If $\alpha$ is a coherent slope at $(x,y)$, then $\alpha^{-1}$ is a coherent slope at $(y,x)$. Moreover, if $s$ and $t$ are codes for $\alpha$ and $\alpha^{-1}$, respectively, then $s$ and $t$ are $\CM$-interalgebraic.
\end{lemma}

\begin{proof} Similar to above, noting that if $\alpha=\partial\left(\frac{x\rightarrow y}{A}\right)$, then $\alpha^{-1}=\partial\left(\frac{y\rightarrow x}{A}\right)$.
\end{proof}

To interpret a field, we will prove yet another analogous result, this time for addition of slopes (and which is not analogous to any statement in \cite{CaHaYe}). Recovering the sum of (independent) coherent slopes is trickier than recovering composition, since there is no obvious $\CM$-definable operation sending two representatives of such slopes to a representative of their sum. To overcome this we use the \textit{two-dimensional} chain rule (i.e. Lemma \ref{L: chain} for $n=2$), combined with an additional technique introduced in \cite{CasACF0}. The idea is that the statement of the higher-dimensional chain rule gives us certain linear combinations of two slopes -- and we can then cancel the coefficients on these linear combinations using a more sophisticated instance of composition.

As in \cite{CasACF0}, after adding constants we may fix $G\subset M^3$, an $\mathcal M(\emptyset)$-definable set of Morley rank 2 whose projections to $M^2$ are all finite-to-one with generic image.

\begin{lemma}\label{L: slope addition}
 Let $(x,y)\in M^2$ be generic, and let $\alpha_1$ and $\alpha_2$ be independent coherent slopes at $(x,y)$. Then $\alpha_3=\alpha_1+\alpha_2$ is a coherent slope, and if $s_i$ is a code of $\alpha_i$ (for $i=1,2,3$), then $s_1s_2s_3$ is coherent and $s_3\in\acl_{\mathcal M}(s_1s_2)$. 
\end{lemma}
\begin{proof}
   Let $u_1\in M$ be generic over all other data. Then there is $u_2$ with $(u_1,u_2,y)\in G$. Let $\beta_1=\partial\left(\frac{u_1\rightarrow y}{u_2}\right)$ and $\beta_2=\partial\left(\frac{u_2\rightarrow y}{u_1}\right)$. Then each $\beta_i$ is an (algebraic) coherent slope at $(u_i,y)$. It follows from Lemmas \ref{L: slope composition} and \ref{L: inverse slope} that $\gamma_i=\beta_i^{-1}\circ\alpha_i$ is a coherent slope at $(x,u_i)$. Let $c_i$ be a coherent representative of $\gamma_i$.
   
   By choosing the $c_i$ as independently as possible, one can ensure that the full tuple $(x,u_1,u_2,y,c_1,c_2)$ is coherent, and $\dim(x,u_1,u_2,y/c_1c_2)=1$. It follows from Lemma \ref{L: independent slope} that each $\partial\left(\frac{x\rightarrow u_i}{c_1c_2}\right)=\gamma_i$ and each $\partial\left(\frac{u_i\rightarrow y}{c_1c_2\{u_j:j\neq i\}}\right)=\beta_i$. So by the chain rule, $$\partial\left(\frac{x\rightarrow y}{c_1c_2}\right)=\beta_1\gamma_1+\beta_2\gamma_2=\alpha_1+\alpha_2=\alpha_3.$$ But from above, $xyc_1c_2$ is coherent, thus $\alpha_3$ is coherent.
   
   Finally, let $s_i$ be a code of each $\alpha_i$. From above, $s_3\in\acl_{\CM}(c_1c_2xy)$. Also note that $\alpha_1=\partial\left(\frac{x\rightarrow y}{c_1u_2}\right)$ (by the chain rule applied to $x\rightarrow u_1\rightarrow y$), and thus $s_1\in\acl_{\CM}(c_1u_2xy)$. Similarly, $s_2\in\acl_{\CM}(c_2u_1xy)$. So in total, $s_1s_2s_3\in\acl_{\CM}(c_1c_2xu_1u_2y)$. As the latter tuple is coherent, so is $s_1s_2s_3$. But $s_3\in\acl(s_1s_2)$ (since $\alpha_3=\alpha_1+\alpha_2$), so by coherence, $s_3\in\acl_{\CM}(s_1s_2)$.
    \end{proof}

\subsection{Finding Infinitely Many Slopes}

At this point, we have shown that the sum and composition operations on coherent slopes can be traced into $\CM$.  However, we have not ruled out the possibility that there is (say) only one coherent slope at each point. In this case, the results above on sum and composition would be trivial, and we would be no closer to interpreting an infinite field.

In this subsection, we point out that $\CM$ \textit{must} realize infinite families of coherent slopes at some generic point -- in other words, that \textit{non-algebraic coherent slopes} exist. This is essentially a consequence of Sard's theorem -- precisely, it will be proved by applying Lemma \ref{L: tangency slope is generic}. 
\begin{lemma}\label{L: slope family exists} There is a non-algebraic coherent slope.
\end{lemma}
\begin{proof}
  By non-local modularity, there are a tuple $(a,b)=(a_1,a_2,b_1,b_2)\in M^4$, and a parameter set $A$, such that:
  \begin{enumerate}
      \item $\rk(a/A)=rk(b/A)=2$.
      \item $\rk(a_1/Ab)=\rk(a_2/Ab)=\rk(a/Ab)=1$.
      \item $b$ is $\CM$-interalgebraic with $\operatorname{Cb}(\operatorname{stp}_{\CM}(a/Ab))$.
    \end{enumerate}
    Replacing with an $\CM$-automorphism conjugate if necessary, we may assume that $Aab$ is coherent. Then $p=\tp(a,b/A)$ is an abstract family of plane curves, and $(a,b,b)$ is a $(p,p)$-intersection. The key point is now:
    \begin{claim}
        $(a,b,b)$ is strongly approximable.
    \end{claim}
    \begin{proof} By a compactness argument, any neighborhood of $b$ contains a realization $b'\models\tp(b/Aa)$ with $b'$ independent from $b$ over $Aa$. It follows easily that $(a,b,b')$ is a strong $(p,p)$-intersection.
    \end{proof}
    By the claim, $(a,b,b)$ is a $(p,p)$-tangency. So by Lemma \ref{L: tangency slope is generic}, $T(a/Ab)\notin\acl(Aa)$. Then $\partial\left(\frac{a_1\rightarrow a_2}{Ab}\right)$ is a non-algebraic coherent slope.
\end{proof}
\textbf{For the rest of this section, we fix a generic $(x,y)\in M^2$ and a non-algebraic coherent slope $\alpha$ at $(x,y)$, with code $s$.} For the rest of the proof, we will essentially be treating $\alpha$ (and thus also $x$, $y$, and $s$) as absorbed constants -- but we will not explicitly do this, because we still need the genericity of $(x,y)$ for certain instances of Lemmas \ref{L: slope composition} and \ref{L: slope addition}. Thus, most statements will involve $\alpha$ as a parameter.

\subsection{Recovering the Operations at One Point}

We now have an infinite family of coherent slopes at a point $(x,y)$, and (assuming codes exist) we can code independent sums and compositions of slopes where they are defined. This is not quite good enough, because one cannot compose two slopes at $(x,y)$ (and forcing $x=y$ would invalidate most of the proofs to this point). Toward this end, our final remaining task is to encode slopes at $(x,x)$ into slopes at $(x,y)$ -- that is, we identify the given family of slopes at $(x,y)$ with a family of $K$-linear endomorphisms of $T(x)$ (noting that the collection of all such endomorphisms forms a $\CK$-definable field, say $F$). We then check that we can still recover the sum and composition operations on $F$ when viewed as operations on slopes at $(x,y)$ through our encoding. Assuming TIMI, this will (essentially) allow us to encode the full field $F$ into $\CM$.

We first define the encoding:

\textbf{For the rest of this section, we fix $F$, the field of $K$-linear endomorphisms of $T(x)$.} Note that $F$ is definable in $\CK$ over $x$ (thus also over $\alpha$).

\begin{definition}
    Let $\beta\in F$ be generic over $\alpha$. We say that $\beta$ is \textit{coherent-generic} if the composition $\alpha\circ\beta:T(x)\rightarrow T(y)$ is a coherent slope at $(x,y)$. In this case, a \textit{code} of $\beta$ is just a code of $\alpha\circ\beta$.
\end{definition}

We thus easily have (in analogy to Lemma \ref{L: slope family exists}):

\begin{lemma}\label{L: coherent field element} There is a coherent-generic element of $F$.
\end{lemma}
\begin{proof} Let $\alpha'\models\tp(\alpha/xy)$ with $\alpha$ and $\alpha'$ independent over $xy$. Then $\beta=\alpha^{-1}\alpha'$ is a coherent-generic element of $F$.
\end{proof}

We now adapt Lemmas \ref{L: slope composition} and \ref{L: slope addition} to this context.

\begin{proposition}\label{P: slope addition at a point}
    Let $\beta_1,\beta_2\in F$ be coherent-generic and independent over $\alpha$. Then $\beta_3=\beta_1+\beta_2$ is coherent-generic, and if $s_i$ is a code of $\beta_i$ for each $i$, then $ss_1s_2s_3$ is coherent and $s_3\in\acl_{\mathcal M}(ss_1s_2)$.
\end{proposition}
\begin{proof} Clear, by noting that $\alpha\beta_3=\alpha\beta_1+\alpha\beta_2$ and applying Lemma \ref{L: slope addition}.
\end{proof}

Composition will be trickier:

\begin{proposition}\label{P: slope composition at a point} Let $\beta_1,\beta_2\in F$ be coherent-generic and independent over $\alpha$. Then $\beta_3=\beta_2\circ\beta_1$ is coherent-generic, and if $s_i$ is a code of $\beta_i$ for each $i$, then $ss_1s_2s_3$ is coherent and $s_3\in\acl_{\mathcal M}(ss_1s_2)$.
\end{proposition}
\begin{proof}
    Let $(\gamma,z)\models\tp_{\mathcal K}(\alpha^{-1},x/y)$ be independent from all other data. So $\gamma$ is a non-algebraic coherent slope at $(y,z)$. Let $t$ be a code of $\gamma$.
    
    Our goal is to show that the slope $\alpha\beta_2\beta_1$ (at $(x,y)$) is coherent. We cannot directly apply Lemma \ref{L: slope composition}, because that lemma requires compositions moving between three different points. Instead, we will break $\alpha\beta_2\beta_1$ into a longer sequence of maps, then regroup them carefully so that Lemma \ref{L: slope composition} can be used at each step. This is similar in spirit to \cite[7.27]{CasACF0}.
    
    Let us proceed. First note that $$\alpha\beta_2\beta_1=(\alpha\beta_2\alpha^{-1}\gamma^{-1})(\gamma\alpha\beta_1)$$ (going from $x$ to $z$ to $y$). Let $\delta_1=\gamma\alpha\beta_1$ and $\delta_2=\alpha\beta_2\alpha^{-1}\gamma^{-1}$, so $\alpha\beta_2\beta_1=\delta_2\delta_1$. Note that $\dim(\delta_1\delta_2/xyz)=2$. Using this, one can check that $\delta_1$ and $\delta_2$ satisfy the various independence assumptions of Lemma \ref{L: slope composition} (this is the main reason for choosing $\delta_1$ and $\delta_2$ this way). 
    
    \begin{claim} Each $\delta_i$ is coherent, and if $d_i$ is a code of $\delta_i$ then $d_i\in\acl_{\mathcal M}(sts_1s_2)$. 
    \end{claim}
    \begin{proof} For $\delta_1$, apply Lemma \ref{L: slope composition} to the composition $\delta_1=(\gamma\alpha)(\beta_1)$. For $\delta_2$, apply Lemmas \ref{L: slope composition} and \ref{L: inverse slope} to the composition $\delta_3=\alpha^{-1}\gamma^{-1}$, and then apply Lemma \ref{L: slope composition} to the composition $\delta_2=(\alpha\beta_2)(\delta_3)$.
    \end{proof}
    
    By the claim, Lemma \ref{L: slope composition} applies to $\alpha\beta_3=\delta_2\delta_1$. We conclude that $\alpha\beta_3$ is coherent, and if $s_3$ is a code (and $d_i$ is a code of $\delta_i$), then $$s_3\in\acl_{\mathcal M}(d_1d_2)\subset\acl_{\mathcal M}(sts_1s_2).$$ But $t$ was chosen independently over all other data, so it can be dropped, yielding $s_3\in\acl_{\mathcal M}(ss_1s_2)$ and proving the proposition.
\end{proof}

\subsection{Building a Field Configuration}

In this subsection, we interpret a field in $\CM$ (assuming TIMI). We will do this by building a \textit{field configuration}. For our purposes, we define:

\begin{definition}\label{D: field con}
    Let $\CN$ be a geometric structure, $A$ a parameter set, and denote dimension in $\CN$ by $d$. An \textit{affine field configuration in $\CN$ over $A$} is a tuple $(a_{12},a_{23},a_{23},u_1,u_2,u_3)$ from $\CN^{eq}$ such that:
    \begin{enumerate}
        \item Each $\dim(a_{ij}/A)=2$ and each $\dim(u_i/A)=1$.
        \item Any two of the six points are independent over $A$, and the dimension of the whole tuple over $A$ is 5.
        \item Each of $a_{12}$, $a_{23}$, $a_{13}$ is algebraic over $A$ with the other two.
        \item For all $i,j\in\{1,2,3\}$ with $i<j$, $u_i$ and $u_j$ are interalgebraic over $Aa_{ij}$.
    \end{enumerate}
\end{definition}

There is a canonical way to build affine field configurations: suppose $\CN$ is a field with geometric theory, and let $AGL_1(N)$ be the group of affine linear bijections $N\rightarrow N$. Let $A$ be any parameter set. Then let $(a_{12},a_{23},u_1)\in(AGL_1(N))^2\times N$ be generic over $A$, and define $a_{13}=a_{23}\circ a_{12}$, $u_2=a_{12}u_1$, and $u_3=a_{23}u_2=a_{13}u_1$. Then $(a_{12},a_{23},a_{13},u_1,u_2,u_3)$ is an affine field configuration in $\CN$ over $A$. Call this a \textit{standard affine field configuration over $A$}.

The main associated fact about affine field configurations is a converse in the strongly minimal case:

\begin{fact}[Hrushovski, see \cite{Bou}] Suppose $\CN$ is strongly minimal and there is a field configuration in $\CN$ over some $A$. Then in $\CN^{eq}$ there is a strongly minimal set admitting an $\CN$-definable (algebraically closed) field structure.
\end{fact}

We will show:

\begin{theorem}\label{T: field interpretation}
    Suppose $\CK$ satisfies TIMI. Then there is an affine field configuration in $\CM$ over some set, and thus there is a strongly minimal set in $\CM^{eq}$ admitting an $\CM$-definable algebraically closed field structure. 
\end{theorem}

Our strategy is as follows: first, we build a standard affine field configuration in our distinguished field $F=\operatorname{Aut}(T(x))$. Then assuming TIMI, we use codes to build a corresponding tuple of points in $\CM^{eq}$, and conclude (repeatedly using Propositions \ref{P: slope addition at a point} and \ref{P: slope composition at a point}) that this tuple from $\CM^{eq}$ is a field configuration in the sense of $\CM$. The key point is to ensure that all slopes we work with are coherent, so that the relevant dimension-theoretic properties in Definition \ref{D: field con} transfer directly from $\CK$ to $\CM$.

Our first step is to show that Propositions \ref{P: slope addition at a point} and \ref{P: slope composition at a point} combine to generically encode the action of the affine linear group of $F$.

\textbf{For the rest of this section, let $AGL_1(F)$ be the group of maps $t\mapsto at+b$ on $F$.}

\begin{definition}
    Let $g\in AGL_1(F)$, so $g(t)=at+b$ for unique $a,b\in F$. We say that $g$ is \textit{coherent-generic} if $a$ and $b$ are coherent-generic and independent over $\alpha$. In this case, a code of $g$ is a pair $s_a,s_b$ where $s_a$ is a code of $a$ and $s_b$ is a code of $b$.
\end{definition}

\begin{lemma}\label{L: AGL action} Let $g,h\in AGL_1(F)$ and $t\in F$ be coherent-generic and independent over $\alpha$.
\begin{enumerate}
    \item $g(t)$ is coherent-generic, and if $s_1,s_2,s_3$ are codes of $g$, $t$, and $g(t)$, respectively, then $s_3\in\acl_{\mathcal M}(ss_1s_2)$.
    \item $g\circ h$ is coherent-generic, and if $s_1,s_2,s_3$ are codes of $g$, $h$, and $gh$, respectively, then $s_3\in\acl_{\mathcal M}(ss_1s_2)$.
\end{enumerate}
\end{lemma}
\begin{proof}
    \begin{enumerate}
        \item Write $g(t)=at+b$, and successively apply Propositions \ref{P: slope composition at a point} and \ref{P: slope addition at a point}.
        \item Write $g(t)=a_gt+b_g$ and $h(t)=a_ht+b_h$. Then $g\circ h$ is the map $t\mapsto(a_ga_h)t+g(b_h)$. Now successively apply Proposition \ref{P: slope composition at a point}, (1), and Proposition \ref{P: slope addition at a point}.
    \end{enumerate}
\end{proof}

We now finish the proof:

\begin{proof}[Proof of Theorem \ref{T: field interpretation}] Assume that $(\CK,\tau)$ satisfies TIMI. By Propositions \ref{P: detection of tangency} and \ref{P: codes exist}, every coherent slope has a code -- and thus so does every coherent-generic element of $F$, as well as every coherent-generic element of $AGL_1(F)$. We use these facts repeatedly.

Let $(g_{12},g_{23},g_{13},t_1,t_2,t_3)$ be a standard affine field configuration in $F$ over $\alpha$. So $(g_{12},g_{23},t_1)$ is generic in $(AGL_1(F))^2\times F$ over $\alpha$, and the others are defined as in the paragraph after Definition \ref{D: field con}. Note that by using Lemma \ref{L: coherent field element} repeatedly to choose $g_{12}$, $g_{23}$, and $t_1$, we may assume they are coherent-generic. It follows inductively by Lemma \ref{L: AGL action} that $g_{13}$, $t_2$, and $t_3$ are also coherent-generic. Let $s_{ij}$ be a code for each $g_{ij}$, and $u_i$ a code for each $t_i$. Let $w=(s_{12},s_{23},s_{13},u_1,u_2,u_3)$. It follows inductively from Lemma \ref{L: AGL action} that $w\in\acl_{\mathcal M}(ss_{12}s_{23}u_1)$ (recall that $s$ is a distinguished code of $\alpha$, so is $\CK$-interalgebraic with $\alpha$). By independence, $ss_{12}s_{23}u_1$ is coherent, and thus so is $sw$. But $sw$ is, point-by-point, $\CK$-interalgebraic with $(\alpha,g_{12},g_{23},g_{13},t_1,t_2,t_3)$. It follows that $(s_{12},s_{23},s_{13},u_1,u_2,u_3)$ forms an affine field configuration in $\mathcal K$ over $s$. By coherence, the same holds in $\mathcal M$.
\end{proof}

Finally, moving back to our T-convex field $\CR_V$, we obtain:

\begin{theorem}\label{T: at least 2} Every strongly minimal 1-dimensional definable $\CR_V$-relic is locally modular.
\end{theorem}
\begin{proof}
Assume not, and let $\CM$ be such a relic which is not locally modular. We may assume that $\CM$ satisfies all of the assumptions made about $\CM$ in this section. By TIMI (Theorem \ref{T: t-convex timi}), we may apply Theorem \ref{T: field interpretation} to conclude that $\CM$ interprets a strongly minimal algebraically closed field. Call this field $F$. By strong minimality, $F$ is in finite correspondence with $M$, implying $\dim(F)=\dim(M)=1$. But $\CR_V$ does not interpret any 1-dimensional algebraically closed fields \cite[Theorem 4.21]{HaHaPe}, a contradiction. Let us elaborate\footnote{Halevi, Peterzil and the second author have recently shown, using similar methods, that any dp-minimal field locally almost embeddable into a weakly o-minimal field is real closed. We give a proof relying only on published results.}: by \cite{JohnDpJML} any dp-minimal unstable field is an SW-uniformity (in the sense of \cite{SimWal}). Moreover, as we have seen, $T$-convex expansions of o-minimal fields have generic differentiability (this is necessary to apply the facts we now use). Now the field $F$ we found is locally almost embeddable into $R$, so it is locally almost strongly internal to $R$ (\cite[Corollary 9.5]{HaHaPeGps}). So \cite[Theorem 4.21]{HaHaPe} applies, implying that $F$ is definably isomorphic to either $R$ or to $R[\sqrt{-1}]$. The latter is 2-dimensional, so we have reached the desired contradiction. 
\end{proof}

As a corollary we obtain a new proof of the main result of \cite{HaOnPe}:
\begin{corollary}
    Let $\CN$ be an o-minimal structure. Let $\CM$ be a 1-dimensional $\CN$-relic. Then $\CM$ is locally modular. 
\end{corollary}
\begin{proof}
    By Theorem \ref{T: reduction2fields} we may assume that $\CN$ is an o-minimal expansion of a real closed field. Passing to an $|\CL|^+$-saturated elementary extension, we may expand $\CN$ by a $T$-convex valuation ring. Let $\CN_V$ be the resulting structure. Then $\CM$ is a strongly minimal definable 1-dimensional $\CN_V$-relic, and the conclusion follows from the previous theorem.  
\end{proof}

\begin{remark} As in Remark \ref{R: imaginaries}, we expect the proof of Theorem \ref{T: at least 2} to adapt to interpretable sorts in power bounded T-convex fields $\CR_V$. Namely, suppose $\CM$ is strongly minimal, not locally modular, and interpreted in $\CR_V$, and $\dim(M)=1$. Then, as in Remark \ref{R: imaginaries}, there is an additive weak rank on $M$; and moreover there should be an infinite interpretable set $S\subset M$ which embeds definably into $R$. Without loss of generality, after adding parameters, we may assume that $S\subset M$ is itself a $\emptyset$-definable subset of $R$. One should then be able to develop the theory of coherent slopes exactly as above, with the added requirement that the coordinates where a coherent slope is taken must lie in $S$. The result should be that $\CM$ interprets a one-dimensional (in the sense of $\CR_V$) algebraically closed field. One can then apply results on interpreted field from \cite{HaHaPe} to reach a similar contradiction.
\end{remark}


\begin{thebibliography}{10}
\bibliographystyle{plain}

\bibitem{AcoLoGps}
Juan~Pablo {Acosta L{\'o}pez}.
\newblock {Closed bounded sets in 1-h-minimal valued fields}.
\newblock {\em arXiv e-prints}, page arXiv:2406.09249, June 2024.

\bibitem{Bou}
Elisabeth Bouscaren.
\newblock The group configuration---after {E}. {H}rushovski.
\newblock In {\em The model theory of groups ({N}otre {D}ame, {IN}, 1985--1987)}, volume~11 of {\em Notre Dame Math. Lectures}, pages 199--209. Univ. Notre Dame Press, Notre Dame, IN, 1989.

\bibitem{CasOmin}
Benjamin {Castle}.
\newblock {The O-minimal Zilber Conjecture in Higher Dimensions}.
\newblock {\em arXiv e-prints}, page arXiv:2406.09285, June 2024.

\bibitem{CasACF0}
Benjamin Castle.
\newblock Zilber’s restricted trichotomy in characteristic zero.
\newblock {\em Journal of the American Mathematical Society}, 37(4):1041--1120, 2024.

\bibitem{cashasva}
Benjamin Castle and Assaf Hasson.
\newblock Very ampleness in strongly minimal sets.
\newblock {\em Model Theory}, 3(2):213--258, 2024.

\bibitem{CaHaYe}
Benjamin Castle, Assaf Hasson, and Jinhe Ye.
\newblock Zilber's trichotomy in hausdorff geometric structures.
\newblock {\em arxiv e-prints}, page arxiv:2405.02209, 2024.

\bibitem{DolGodViscerality}
Alfred Dolich and John Goodrick.
\newblock Tame topology over definable uniform structures, 2021.

\bibitem{vdDriesGpChunk}
L.~P.~D. van~den Dries.
\newblock Weil's group chunk theorem: a topological setting.
\newblock {\em Illinois J. Math.}, 34(1):127--139, 1990.

\bibitem{vdDries-Tconvex}
Lou van~den Dries.
\newblock {$T$}-convexity and tame extensions. {II}.
\newblock {\em J. Symbolic Logic}, 62(1):14--34, 1997.

\bibitem{vdDries-TconvexCorrect}
Lou van~den Dries.
\newblock Correction to: ``{$T$}-convexity and tame extensions. {II}'' [{J}. {S}ymbolic {L}ogic {\bf 62} (1997), no. 1, 14--34; {MR}1450511 (98h:03048)].
\newblock {\em J. Symbolic Logic}, 63(4):1597, 1998.

\bibitem{vdDriesLewen}
Lou van~den Dries and Adam~H. Lewenberg.
\newblock {$T$}-convexity and tame extensions.
\newblock {\em J. Symbolic Logic}, 60(1):74--102, 1995.

\bibitem{Gagelman}
Jerry Gagelman.
\newblock Stability in geometric theories.
\newblock {\em Ann. Pure Appl. Logic}, 132(2-3):313--326, 2005.

\bibitem{HaHaPe}
Yatir {Halevi}, Assaf {Hasson}, and Ya'acov {Peterzil}.
\newblock {Fields interpretable in $P$-minimal fields}.
\newblock {\em arXiv e-prints}, page arXiv:2103.15198, March 2021.

\bibitem{HaHaPeGps}
Yatir {Halevi}, Assaf {Hasson}, and Ya'acov {Peterzil}.
\newblock {On groups interpretable in various valued fields}.
\newblock {\em arXiv e-prints}, page arXiv:2206.05677, June 2022.

\bibitem{HaHrMac3}
Deirdre Haskell, Ehud Hrushovski, and Dugald Macpherson.
\newblock Unexpected imaginaries in valued fields with analytic structure.
\newblock {\em J. Symbolic Logic}, 78(2):523--542, 2013.

\bibitem{HaOnPe}
Assaf Hasson, Alf Onshuus, and Ya'acov Peterzil.
\newblock {Definable one dimensional structures in o-minimal theories}.
\newblock {\em Israel J. Math.}, 179:297--361, 2010.

\bibitem{HaOnPi}
Assaf {Hasson}, Alf {Onshuus}, and Santiago {Pinzon}.
\newblock {Strongly minimal group relics of algebraically closed valued fields}.
\newblock {\em arXiv e-prints}, page arXiv:2401.14618, January 2024.

\bibitem{HruPil}
Ehud Hrushovski and Anand Pillay.
\newblock Groups definable in local fields and pseudo-finite fields.
\newblock {\em Israel J. Math.}, 85(1-3):203--262, 1994.

\bibitem{HrPi2011}
Ehud Hrushovski and Anand Pillay.
\newblock Affine {N}ash groups over real closed fields.
\newblock {\em Confluentes Math.}, 3(4):577--585, 2011.

\bibitem{JohnDpJML}
Will Johnson.
\newblock The canonical topology on dp-minimal fields.
\newblock {\em J. Math. Log.}, 18(2):1850007, 23, 2018.

\bibitem{JohnEOI}
Will Johnson.
\newblock Interpretable sets in dense o-minimal structures.
\newblock {\em J. Symb. Log.}, 83(4):1477--1500, 2018.

\bibitem{JohnQpTop}
Will Johnson.
\newblock Topologizing interpretable groups in {$p$}-adically closed fields.
\newblock {\em Notre Dame J. Form. Log.}, 64(4):571--609, 2023.

\bibitem{JohnVisc}
Will {Johnson}.
\newblock {Visceral theories without assumptions}.
\newblock {\em arXiv e-prints}, page arXiv:2404.11453, April 2024.

\bibitem{KowRand}
P.~{Kowalski} and S.~{Randriambololona}.
\newblock {Strongly minimal reducts of valued fields}.
\newblock {\em J Symbolic logic}, August 2014.
\newblock to appear.

\bibitem{MarikovaGps}
Jana Ma\v{r}\'{\i}kov\'{a}.
\newblock Type-definable and invariant groups in o-minimal structures.
\newblock {\em J. Symbolic Logic}, 72(1):67--80, 2007.

\bibitem{MeRuSt}
Alan Mekler, Matatyahu Rubin, and Charles Steinhorn.
\newblock {Dedekind completeness and the algebraic complexity of {$o$}-minimal structures}.
\newblock {\em Canad. J. Math.}, 44(4):843--855, 1992.

\bibitem{MelRCVFEOI}
T.~Mellor.
\newblock Imaginaries in real closed valued fields.
\newblock {\em Ann. Pure Appl. Logic}, 139(1-3):230--279, 2006.

\bibitem{PeStTricho}
Ya'acov Peterzil and Sergei Starchenko.
\newblock A trichotomy theorem for o-minimal structures.
\newblock {\em Proc. London Math. Soc. (3)}, 77(3):481--523, 1998.

\bibitem{Pi5}
Anand Pillay.
\newblock {On groups and fields definable in {$o$}-minimal structures}.
\newblock {\em J. Pure Appl. Algebra}, 53(3):239--255, 1988.

\bibitem{SimWal}
Pierre Simon and Erik Walsberg.
\newblock Tame topology over dp-minimal structures.
\newblock {\em Notre Dame J. Form. Log.}, 60(1):61--76, 2019.



\end{thebibliography}
\end{document}